\documentclass[reqno]{amsart}
\usepackage{amssymb}
\usepackage[mathscr]{eucal}
\usepackage{hyperref}

\theoremstyle{plain}
\newtheorem{prop}{Proposition}[section]
\newtheorem{thm}[prop]{Theorem}
\newtheorem{cor}[prop]{Corollary}
\newtheorem{lem}[prop]{Lemma}

\theoremstyle{definition}
\newtheorem{dfn}[prop]{Definition}

\newtheorem{example}[prop]{Example}
\newtheorem{lab}[prop]{}

\newcommand{\isoto}{\overset{\sim}{\to}}

\newcommand{\A}{{\mathbb{A}}}

\newcommand{\R}{{\mathbb{R}}}

\newcommand{\m}{{\mathfrak{m}}}

\newcommand{\calA}{{\mathcal{A}}}

\newcommand{\calI}{{\mathcal{I}}}
\newcommand{\calO}{{\mathcal{O}}}
\newcommand{\calT}{{\mathcal{T}}}

\newcommand{\scrD}{{\mathscr{D}}}
\newcommand{\scrO}{{\mathscr{O}}}

\DeclareMathOperator{\rk}{rk}
\DeclareMathOperator{\sosx}{sosx}
\DeclareMathOperator{\sxdeg}{sxdeg}
\DeclareMathOperator{\spn}{span}

\newcommand{\id}{\mathrm{id}}
\newcommand{\ord}{{\rm ord}}
\newcommand{\Sym}{\mathrm{Sym}}

\newcommand{\itLambda}{{\mathit{\Lambda}}}
\newcommand{\itOmega}{{\mathit{\Omega}}}
\newcommand{\bfmu}{{\boldsymbol{\mu}}}

\newcommand{\ula}{{\underline a}}
\newcommand{\ulb}{{\underline b}}
\newcommand{\ulm}{{\underline m}}
\newcommand{\ulp}{{\underline p}}
\newcommand{\ulq}{{\underline q}}
\newcommand{\ult}{{\underline t}}
\newcommand{\ulxi}{{\underline\xi}}
\newcommand{\uleta}{{\underline\eta}}

\newcommand{\comp}{\mathbin{\scriptstyle\circ}}
\renewcommand{\emptyset}{\varnothing}
\renewcommand{\setminus}{\smallsetminus}
\newcommand{\ol}{\overline}
\newcommand{\wh}[1]{\widehat{#1}}
\newcommand{\all}{\forall\,}
\newcommand{\ex}{\exists\,}
\newcommand{\To}{\Rightarrow}
\renewcommand{\subset}{\subseteq}
\renewcommand{\supset}{\supseteq}
\newcommand{\idl}[1]{\langle #1\rangle}

\newcommand{\sa}{semialgebraic}

\begin{document}

\title
{Convex hulls of curves in $\boldsymbol n$-space}

\author{Claus Scheiderer}
\address
  {Fachbereich Mathematik und Statistik \\
  Universit\"at Konstanz \\
  78457 Konstanz \\
  Germany}
\email
  {claus.scheiderer@uni-konstanz.de}

\begin{abstract}
Let $K\subset\R^n$ be a convex \sa\ set. The semidefinite extension
degree $\sxdeg(K)$ of $K$ is the smallest number $d$ such that $K$ is
a linear image of an intersection of finitely many spectrahedra, each
of which is described by a linear matrix inequality of size $\le d$.
This invariant can be considered to be a measure for the intrinsic
complexity of semidefinite optimization over the set $K$. For an
arbitrary \sa\ set $S\subset\R^n$ of dimension one, our main result
states that the closed convex hull $K$ of $S$ satisfies
$\sxdeg(K)\le1+\lfloor\frac n2\rfloor$. This bound is best possible
in several ways. Before, the result was known for $n=2$, and also for
general $n$ in the case when $S$ is a monomial curve.
\end{abstract}

\date\today
\maketitle


\section*{Introduction}

Semidefinite programming (SDP) is the task of optimizing a linear
function over the solution set of a linear matrix inequality (LMI)
$$A_0+\sum_{i=1}^nx_iA_i\>\succeq\>0$$
where $A_0,\dots,A_n$ are real symmetric matrices of some size
and $A\succeq0$ means that $A$ is positive semidefinite. Under mild
conditions, semidefinite programs can be solved in polynomial time,
up to any prescribed accuracy. Thanks to its enormous expressive
power, semidefinite programming has numerous applications from a wide
range of areas. See \cite{av} for detailed background on SDP.

Solution sets of linear matrix inequalities are called spectrahedra.
The feasible sets of semidefinite programming are therefore
spectrahedra and, more generally, linear images of
spectrahedra (aka spectrahedral shadows). The performance of
numerical SDP solvers is known to be heavily influenced by the
matrix size of the LMI. On the other hand, it is often possible in
concrete examples to represent a given convex set $K$ by a
combination of several LMIs of small size~$d$. Practical
experience shows that this size $d$ is far more critical for the
running time than the number of the LMIs. Motivated by these
observations, Averkov \cite{av} introduced an invariant of convex
\sa\ sets $K\subset\R^n$ that captures this essential bit of the
intrinsic complexity of semidefinite optimization over $K$:
The \emph{semidefinite extension degree} of $K$, denoted
$\sxdeg(K)$, is the smallest number $d$ such that $K$ can be written
as a linear image of a finite intersection of spectrahedra, all of
which are described by LMIs of size $\le d$.
So $\sxdeg(K)\le d$ means that $K$ can be represented in the form
$$K\>=\>\bigl\{x\in\R^n\colon\ex y\in\R^m\ A(x,y)\succeq0\bigr\}$$
for some $m$, where $A(x,y)=A_0+\sum_{i=1}^nx_iA_i+\sum_{j=1}^m
y_jB_j$ is a symmetric linear matrix polynomial of block-diagonal
structure with all blocks of size at most~$d$. For example,
$\sxdeg(K)\le1$ if and only if $K$ is a polyhedron, and
$\sxdeg(K)\le2$ if and only if $K$ is second-order cone
representable. It is this invariant that we are going to study.

Let $n\ge1$, and let $S$ be an arbitrary \sa\ set in $\R^n$ of
dimension one. Our main result (Theorem \ref{mainthm}) states that
the closed convex hull $K$ of $S$ satisfies
$\sxdeg(K)\le1+\lfloor\frac n2\rfloor$. Before, this was known for
$n=2$ \cite{sch:sxdeg}, and also for general $n$ in the case where
$S$ is a curve parametrized by monomials \cite{avsch}. In the general
case, it was known from \cite{sch:curv} that $\sxdeg(K)$ is finite,
or in other words, that $K$ is a spectrahedral shadow.

The upper bound in Theorem \ref{mainthm} is best possible: If $n=2k$
is even and $K$ is the closed convex hull of the rational normal
curve $\{(t,t^2,\dots,t^n)\colon t\in\R\}$ of degree~$n$, then
$\sxdeg(K)=1+\frac n2$, as follows from \cite[Cor.~2.3]{av}. In this
particular case, $K$ happens to be a spectrahedron, and an explicit
representation of minimal matrix size $k+1$ is given by the LMI
$$\begin{pmatrix}1&x_1&\cdots&x_k\\x_1&x_2&\cdots&x_{k+1}\\
\vdots&\vdots&&\vdots\\x_k&x_{k+1}&\cdots&x_{2k}\end{pmatrix}
\ \succeq\ 0.$$
Our theorem cannot be extended to convex hulls of sets of dimension
greater than one. Indeed, such convex hulls need not be spectrahedral
shadows in general. Explicit examples in $\R^n$ are currently known
for $n\ge11$, and can be found in \cite{sch:hn} or in
\cite[Sect.~8.7]{Sch}.

For monomial curves, the proof of our main result in \cite{avsch}
leads to an explicit semidefinite representation of the convex hull.
This is not so in the general situation considered here. The
basic approach is the same as for monomial curves, but the technical
details are getting considerably more involved in the general case.
To establish the upper bound for $\sxdeg(K)$ we dualize, meaning
that, instead of the convex hull $K$, we study the convex cone of
linear polynomials that are non-negative on $S$. The main tool for
our proof is an algebraic characterization of $\sxdeg(K)$ that uses
the concept of \emph{tensor evaluation}. For precise details we
refer to \ref{tenseval} and Theorem \ref{critenseval} below.

In the (explicit) case of monomial curves, a key role was played in
\cite{avsch} by Schur polynomials and by Jacobi's bialternant
identity. We would like to point out that the same is true in the
general case considered here. Speaking very loosely, the problem can
be localized to some extent. Locally around a given point, Schur
polynomials provide a lowest degree approximation of the problem.
A~large part of the effort then consists in showing that this
approximation dominates the picture in a sufficiently small
neighborhood.

The paper is organized as follows. After stating the main result and
its dual version in Section~\ref{mainreslts}, we apply a number of
reduction steps in Section \ref{firstreducts}. Here we also recall
the concept of tensor evaluation. Sections~\ref{schurprep} and
\ref{extprep} are of preparatory character, dealing respectively with
Schur polynomials and with extreme rays in cones of non-negative
polynomials. A very rough general outline of the proof is given in
Section~\ref{outline}. Sections \ref{repdet1} and \ref{repdet2} set
up the technical machinery needed for the proof proper, while the
actual proof is given in Section~\ref{sectpfmainthm}.


\section{Main result}\label{mainreslts}%

\begin{lab}\label{basixconv}%
Before stating the main theorem we briefly recall a few basic notions
from convexity. Let $C$ be a convex cone in a finite-dimensional
$\R$-vector space $V$, meaning that $C\ne\emptyset$, $C+C\subset C$
and $aC\subset C$ for every real number $a\ge0$. The cone $C$ is said
to be pointed if $C\cap(-C)=\{0\}$. A convex subset $F\ne\emptyset$
of $C$ is a \emph{face} of $C$ if $x,\,y\in C$ and $x+y\in F$ imply
$x,\,y\in F$. If $0\ne x\in C$ is such that the half-line
$F=\{ax\colon a\ge0\}$ is a face of $C$ then $F$ is called an
\emph{extreme ray} of~$C$. Every closed and pointed convex cone is
the Minkowski sum of its extreme rays.
For every $x\in C$ there is a unique smallest face of $C$, denoted
$F_x$, that contains $x$. We call $F_x$ the \emph{supporting face}
of $x$ (in $C$). See \cite[Sect.~8.1]{Sch} for these facts and for
more general background.
\end{lab}

\begin{dfn}\label{dfnsxdeg}%
(See \cite{av})
Let $K$ be a convex \sa\ set in $\R^n$. The \emph{semidefinite
extension degree} $\sxdeg(K)$ of $K$ is the smallest integer $d\ge0$
such that $K$ has a representation
$$K\>=\>\Bigl\{x\in\R^n\colon\ex y\in\R^m\ \all\nu=1,\dots,r\
A_\nu+\sum_{i=1}^nx_iB_{\nu i}+\sum_{j=1}^my_jC_{\nu j}\succeq0
\Bigr\}$$
with $r,\,m\ge1$ and with $A_\nu$, $B_{\nu i}$, $C_{\nu i}$ real
symmetric matrices of the same size $d_\nu\le d$, for all
$\nu=1,\dots,r$. We put $\sxdeg(K)=\infty$ if there is no such
integer $d\ge0$.
\end{dfn}

If $K$ is an affine space then $\sxdeg(K)=0$ (this may be considered
part of the definition), otherwise $\sxdeg(K)\ge1$. It follows
directly that $\sxdeg(K)\le1$ if and only if $K$ is a polyhedron.
Moreover it is not hard to see (e.g.\ \cite[Example 1.5]{sch:sxdeg})
that $\sxdeg(K)\le2$ if and only if $K$ is second-order cone
representable. The main result of this paper is the following:

\begin{thm}\label{mainthm}%
Let $S\subset\R^n$ be a \sa\ set of dimension one, let
$K\subset\R^n$ be the closed convex hull of $S$. Then
$\sxdeg(K)\le1+\lfloor\frac n2\rfloor$.
\end{thm}

Here, of course, $\lfloor\frac n2\rfloor$ denotes the largest
integer $\le\frac n2$. In particular:

\begin{cor}\label{dim3socr}%
The closed convex hull of any one-dimensional \sa\ set in $\R^3$
is second-order cone representable.
\qed
\end{cor}

Let $S\subset\R^n$ be a subset, let $K$ be its closed convex hull and
let
$$P_S\>=\>\{f\in\R[x]_{\le1}\colon f|_S\ge0\},$$
the set of linear polynomials in $x=(x_1,\dots,x_n)$ that are
non-negative on $S$. Recall that the homogenization $K^h$ of $K$ is
the closure of $\{(t,tu)\colon u\in K$, $t\ge0\}$ in $\R\times\R^n$
\cite[Prop.\ 8.1.19]{Sch}. The set $P_S$ is a closed convex cone in
$\R[x]_{\le1}$, and is linearly isomorphic to the dual convex cone of
$K^h$ in a natural way \cite[Example 8.1.23]{Sch}.
By \cite[Cor.~1.9]{sch:sxdeg} we have $\sxdeg(K)=\sxdeg(P_S)$.
So Theorem \ref{mainthm} follows from the following theorem:

\begin{thm}\label{mainthm2}%
Let $C$ be an affine algebraic curve over $\R$, let $S\subset C(\R)$
be a \sa\ set, and let $V$ be a linear subspace of the affine
coordinate ring $\R[C]$ of $C$ with $\dim(V)=n+1$. Then the closed
convex cone
$$P_{V,S}\>:=\>\{f\in V\colon f|_S\ge0\}$$
in $V$ satisfies $\sxdeg(P_{V,S})\le1+\lfloor\frac n2\rfloor$.
\end{thm}

\noindent
\emph{Proof} of \ref{mainthm2} $\To$ \ref{mainthm}:
Assuming Theorem \ref{mainthm2}, let $S\subset\R^n$ be a \sa\ set
of dimension one. We may assume that $S$ is not contained in an
affine hyperplane of $\R^n$. Let $C$ be the Zariski closure of $S$
and let $V\subset\R[C]$ be the image of $\R[x_1,\dots,x_n]_{\le1}$
in $\R[C]$. Then $\dim(V)=n+1$ and $P_{V,S}=P_S$, and so
$\sxdeg(K)=\sxdeg(P_S)=\sxdeg(P_{V,S})$.
\qed


\section{First reduction steps}\label{firstreducts}%

We start by making a couple of first reductions towards a proof of
Theorem \ref{mainthm2}. On one hand they concern the curve $C$ and
the set $S$. On the other we are going to reformulate the task of
bounding $\sxdeg(K)$ in a way that is more abstract and algebraic,
but also more manageable.

\begin{lab}\label{1streducts}%
The set $S$ can be
assumed to be closed.
It is enough to cover $S$ by finitely many \sa\ sets $S_i$ and to
prove the theorem for each of them. Indeed, since
$P_{V,S}=\bigcap_iP_{V,S_i}$ we have $\sxdeg(P_{V,S})\le\max_i
\sxdeg(P_{V,S_i})$ \cite[Lemma 1.4(d)]{sch:sxdeg}. In this way
we may assume that the Zariski closure $C\subset\A^n$ of $S$ is an
irreducible curve, and that $S$ does not have any isolated point.
We may in fact assume that the curve $C$ is non-singular, since
otherwise we replace $C$ by its normalization $C'$ and
$S\subset C(\R)$ by its preimage in $C'(\R)$. Since $S$ does not
have isolated points, $S$ is contained in the image of the natural
map $C'(\R)\to C(\R)$.
\end{lab}

\begin{lab}\label{reductnonsing}%
So let the curve $C$ be irreducible and non-singular. By the previous
discussion we assume that $S$ is closed in $C(\R)$, and is
homeomorphic either to a compact non-degenerate interval or to a
closed half-line.
In fact it is enough to consider compact intervals. Indeed,
assuming that $S$ is homeomorphic to $\left[0,\infty\right[$,
let $\ol C$ be the non-singular projective model of $C$ and let
$u\in\ol C(\R)\setminus C(\R)$ be the point in the closure of $S$.
Choose a point $v\in C(\R)\setminus S$ and let
$C'=\ol C\setminus\{v\}$. Then $C'$ is an affine curve,
and the
subset $S':=S\cup\{u\}$ of $C'(\R)$ is homeomorphic to $[0,1]$. If
$0\ne p\in\R[C']$ is chosen such that $p$ vanishes in all (real and
nonreal) points of $\ol C\setminus C$, and if $k\ge1$ is large
enough, the subspace $V':=p^{2k}V$ of the function field
$\R(C)=\R(C')$ is contained in $\R[C']$.
Multiplication by $p^{2k}$ is a vector space isomorphism $V\to V'$
under which the cone $P_{V,S}$ gets identified with the cone
$P_{V',S'}$.
So it suffices to consider the latter cone.
\end{lab}

Summarizing the reduction steps discussed so far, we have:

\begin{lem}\label{reductpart1}%
To prove Theorem \ref{mainthm2}, it suffices to consider the case
where the curve $C$ is irreducible and non-singular and where the set
$S\subset C(\R)$ is homeomorphic to $[0,1]$.
\qed
\end{lem}

\begin{lab}\label{erklinterval}%
Let $\xi\in C(\R)$ be an $\R$-point. The order of vanishing of a
rational function $f\in\R(C)$ in the point $\xi$ will be denoted
$\ord_\xi(f)$. A \emph{local orientation} of $C(\R)$ at $\xi$ is an
equivalence class of local uniformizers at $\xi$ (elements
$u\in\R[C]$ with $\ord_\xi(u)=1$), where $u_1$ and $u_2$ are said to
be equivalent if $\frac{u_1}{u_2}(\xi)>0$. Let such a local
orientation at $\xi$ be given. By a \emph{closed interval on the
positive side} of $\xi$, we mean a compact connected set
$J\subset C(\R)$ with $\xi\in J$ and $|J|>1$ such that $u|_J\ge0$ for
some local uniformizer $u$ at $\xi$ that represents the given
orientation. Such $J$ has boundary points $\xi$ and $\xi'\ne\xi$, and
we'll simply write $J=[\xi,\xi']$ (tacitly assuming the given local
orientation to be understood).
\end{lab}

With this terminology we may reduce the task of proving Theorem
\ref{mainthm2} still a bit further:

\begin{lem}\label{reductevenmore}%
Let $C$ be an affine curve over $\R$, irreducible and non-singular,
and let $V\subset\R[C]$ be a linear subspace with $\dim(V)=n+1$.
Suppose that, for every $\xi\in C(\R)$ and every local orientation
of $C(\R)$ at~$\xi$, some closed interval $J\subset C(\R)$ on the
positive side of $\xi$ can be found such that
$\sxdeg(P_{V,J_1})\le1+\lfloor\frac n2\rfloor$ holds for every closed
interval $J_1=[\xi,\xi_1]$ contained in $J$. Then
$\sxdeg(P_{V,S})\le1+\lfloor\frac n2\rfloor$ holds for every \sa\
set $S\subset C(\R)$.
\end{lem}

\begin{proof}
By Lemma \ref{reductpart1} we may assume that $S$ is homeomorphic to
$[0,1]$. Then $S$ is covered by finitely many closed intervals
$J_i=[\xi_i,\xi'_i]$ as in the statement, and so
$$\sxdeg(P_{V,S})\>\le\>\max_i\,\sxdeg(P_{V,J_i})\>\le\>
1+\lfloor\frac n2\rfloor$$
by the remarks in \ref{1streducts}.
\end{proof}

\begin{lab}\label{addlasspts2}%
Throughout the paper we'll denote the affine coordinate ring of the
curve $C$ by $A=\R[C]$. In view of Lemma \ref{reductevenmore}, fix a
point $\xi\in C(\R)$ together with one of the two local orientations
at~$\xi$. To verify the condition in this lemma we may replace the
curve $C$ by an arbitrary Zariski-open neighborhood of $\xi$ in $C$.
In this way we may assume that there exists an element $t$ in $A$
such that the $A$-module $\Omega_{A/\R}$ of K\"ahler differentials is
freely generated by $dt$.
After replacing $t$ with $t-t(\xi)$, the element $t$ will be a local
uniformizer at $\xi$, and after multiplication with $-1$ if
necessary, $t$ will represent the given local orientation at~$\xi$.
Both $\xi$ and $t$ will be fixed for the rest of the paper. By
shrinking $C$ further around $\xi$ if necessary, we can in addition
assume that the maximal ideal $\m_\xi=\{f\in A\colon f(\xi)=0\}$ in
$A$ is principal, generated by~$t$.
\end{lab}

\begin{lab}\label{tenseval}%
The second reduction step is more abstract in nature. First we need
to recall some facts from \cite[Sect.~3]{sch:sxdeg} about tensor
evaluation. Let $X$ be an affine $\R$-variety and let $R\supset\R$ be
a real closed field. Write $R[X]=\R[X]\otimes R$ (tensor product
over~$\R$) and let $f\in R[X]$. If $\eta\in X(R)$ is an $R$-rational
point of $X$ (i.e., a homomorphism $\eta\colon\R[X]\to R$), let the
\emph{tensor evaluation} $f^\otimes(\eta)$ (of $f$ at~$\eta$) be the
image of $f$ under the ring homomorphism
$$R[X]\>=\>\R[X]\otimes R\xrightarrow{\eta\otimes\id}R\otimes R.$$
Note that $f(\eta)\in R$ (the usual evaluation of $f$ at $a$) is the
result of applying the product map $R\otimes R\to R$ to
$f^\otimes(a)$.

For $M$ a \sa\ subset of $X(\R)$, let $M_R$ denote the base field
extension of $M$ to $R$ (see \cite[Sect.~4.1]{Sch}). This is the
subset of $X(R)$ that is described by the same finite system of
polynomial inequalities as~$M$. Given $\theta\in R\otimes R$, let
$\rk(\theta)$ (the tensor rank of $\theta$) be the minimal number
$r\ge0$ such that $\theta=\sum_{i=1}^ra_i\otimes b_i$ for suitable
$a_i,\,b_i\in R$. If $\theta$ is a sum of squares in $R\otimes R$,
let $\sosx(\theta)$ be the smallest integer $d\ge0$ such that there
is an identity $\theta=\sum_i\theta_i^2$ with $\rk(\theta_i)\le d$
for all~$i$.
If $\theta$ is not a sum of squares we put $\sosx(\theta)=\infty$.
To prove our main theorem we are going to apply the following
criterion. It is a particular case of \cite[Thm.~3.10]{sch:sxdeg}:
\end{lab}

\begin{thm}\label{critenseval}%
Let $C,\,V,\,S$ and $P=P_{V,S}$ be as in Theorem \ref{mainthm2}, and
let $E$ denote the set of elements in $P$ that span an extreme ray
of $P$. Let $d\ge1$ be an integer. Then $\sxdeg(P)\le d$ holds if and
only if $\sosx f^\otimes(\eta)\le d$ for every real closed field
$R\supset\R$, every $f\in E_R$ and every $\eta\in S_R$.
\end{thm}

\begin{lab}\label{reduct2sosx}%
Summarizing, we have reduced the task of proving Theorem
\ref{mainthm2} to the following. Let $C$ be an irreducible and
non-singular affine curve over $\R$, let $V\subset A=\R[C]$ be a
linear subspace of dimension $n+1$ and let a point $\xi\in C(\R)$ be
given, together with a local orientation of $C(\R)$ at $\xi$. We
need to find a closed interval $S\subset C(\R)$ on the positive side
of $\xi$ such that the following holds: Whenever $R\supset\R$ is real
closed, $\eta\in S_R$ is a point and $f\in V_R=V\otimes R$ is
non-negative on $S_R$ and generates an extreme ray in
$(P_{V,S})_R=\{g\in V_R\colon g|_{S_R}\ge0\}$, there is an identity
$f^\otimes(\eta)=\sum_i\theta_i^2$ in $R\otimes R$ where each
$\theta_i$ has the form $\theta_i=\sum_{j=1}^ka_{ij}\otimes b_{ij}$
with $k=1+\lfloor\frac n2\rfloor$ and suitable
$a_{ij},\,b_{ij}\in R$.
\end{lab}


\section{Complements on Schur polynomials}\label{schurprep}%

\begin{lab}\label{schurpolnot}%
Let $n\ge0$, let $x=(x_0,\dots,x_n)$ be a tuple of $n+1$ variables,
and let $\ula=(a_0,\dots,a_n)$ be a sequence of non-negative integers
that is strictly decreasing, i.e.\ satisfies
$a_0>\cdots>a_n\ge0$. By $\sigma_\ula=\sigma_\ula(x)$ we denote
the polynomial in the variables $x$ that is defined by the identity
\begin{equation}\label{bialtfml}%
\sigma_\ula(x_0,\dots,x_n)\cdot\prod_{0\le i<j\le n}(x_i-x_j)\>=\>
\det\Bigl(\bigl(x_i^{a_j}\bigr)_{i,j=0,\dots,d}\Bigr).
\end{equation}
The polynomial $\sigma_\ula(x)$ is the Schur polynomial associated
with the partition $\lambda=\ula-\delta_n$ where
$\delta_n=(n,\dots,1,0)$, and identity \eqref{bialtfml} is known as
Jacobi's bialternant formula (see \cite[Sect.~7.15]{St2}). Standard
notation would be $s_\lambda(x)$, rather than $\sigma_\ula(x)$, and
would use variables $x_1,\dots,x_n$ instead of $x_0,\dots,x_n$. In
this paper, however, it is the sequence $\ula$ of length $n+1$,
rather than the partition $\lambda$, that plays the main role.
Therefore we prefer the alternative notation used here.
\end{lab}

\begin{lab}\label{combinatschur}%
The polynomial $\sigma_\ula(x)$ is symmetric, and its coefficients
are non-negative integers that allow a combinatorial
characterization. Let us briefly recall this. Let $Y_\ula$ denote the
Young diagram of shape $\lambda=\ula-\delta_n$. So $Y_\ula$ consists
of $n+1$ rows that contain $a_0-n,a_1-n+1,\dots,a_n$ many entries,
respectively, and that are aligned to the left. An admissible filling
$T$ of $Y_\ula$ assigns to each position in $Y_\ula$ an integer from
$\{0,\dots,n\}$, in such a way that the entries are weakly increasing
in each row (from left to right) and strictly increasing in each
column (from top to bottom). If the entry $i$ is assigned to exactly
$\alpha(i)$ positions by the filling $T$, write
$x^T=x_0^{\alpha(0)}\cdots x_n^{\alpha(n)}$. With this notation one
has
$$\sigma_\ula(x)\>=\>\sum_Tx^T,$$
sum over all admissible fillings $T$ of $Y_\ula$,
see \cite[Sect.~7.10]{St2}.
\end{lab}

\begin{lem}\label{schurlem1a}%
Let $\ula\ne\ulb$ be strictly decreasing sequences of non-negative
integers, both of length $n+1$, and assume that $b_i\ge a_i$ for
$i=0,\dots,n$. Then every monomial of $\sigma_\ulb(x)$ is properly
divisible by some monomial of $\sigma_\ula(x)$.
\end{lem}

\begin{proof}
It suffices to check this in the case when there is an index $k$
with $b_k=a_k+1$ and $b_i=a_i$ for all $i\ne k$.
Given an admissible
filling $T$ of $Y_\ulb$, corresponding to a monomial $x^\beta$
of $\sigma_\ulb(x)$, we may simply delete the last entry of $T$ in
row~$k$. This gives an admissible filling of $Y_\ula$ whose
corresponding monomial divides $x^\beta$.
\end{proof}

\begin{lem}\label{schurlem1}%
Let $\ula=(a_0,\dots,a_n)$ be a strictly decreasing sequence of
non-negative integers,
and let $\ulb=(b_0,\dots,b_r)$ be a subsequence of $\ula$ (i.e.,
there are $0\le i(0)<\cdots<i(r)\le n$ with $b_\nu=a_{i(\nu)}$ for
$\nu=0,\dots,r$). Then every monomial of
$\sigma_\ulb(x_0,\dots,x_r)$ is divisible by a monomial of
$\sigma_\ula(x_0,\dots,x_r,1,\dots,1)$.
\end{lem}

\begin{proof}
It suffices to prove this in the case $r=n-1$. Write
$x=(x_0,\dots,x_n)$ and $x'=(x_0,\dots,x_{n-1})$.
Let $Y_\ula$ and $Y_\ulb$ be the Young diagrams of $\ula$ and $\ulb$,
respectively, and let $k\in\{0,\dots,n\}$ be the index for which
$a_k$ is missing in $\ulb$. The bottom $n-k$ rows of $Y_\ula$ and of
$Y_\ulb$ have the same lengths, while the top $k$ rows of $Y_\ula$
are each shorter by one than the corrsponding rows of $Y_\ulb$.
Given an admissible filling $T$ of $Y_\ulb$, corresponding to a
monomial $x'^\beta$ of $\sigma_\ulb(x')$, construct an admissible
filling of $Y_\ula$ in the following way. First delete the last entry
in each of the top $k$ rows of $T$. Then, for $i$ from $0$ to
$n-1$, put the remaining entries of row~$i$ of $T$ into row~$i$ of
$Y_\ula$, starting from the left in each row.
Finally, put the entry~$n$ into those boxes of $Y_\ula$ that are
still unfilled.
The filling of $Y_\ula$ constructed in this way is admissible and
corresponds to a monomial $x^\alpha$ of $\sigma_\ula(x)$ which, after
substitution $x_n=1$, divides $x'^\beta$.
\end{proof}

The following lemma will be needed in Section~\ref{repdet1}:

\begin{lem}\label{detlem2}%
Let $B$ be a commutative ring, let $p_0,\dots,p_n\in B[[t]]$ be
formal power series and let $0\le r\le n$. Let $M$ be a matrix of
size $(n+1)\times(n+1)$ with coefficients in $B[[x_0,\dots,x_r]]$
whose $i$-th row is
$$\bigl(p_0(x_i),\,\dots,\,p_n(x_i)\bigr)$$
for $i=0,\dots,r$, and whose lower $n-r$ rows have coefficients
in~$B$.
\begin{itemize}
\item[(a)]
There is a (unique) power series $g\in B[[x_0,\dots,x_r]]$ such that
$$\det(M)\>=\>g\cdot\prod_{0\le i<j\le r}(x_i-x_j).$$
\item[(b)]
If the vanishing orders $m_i=\ord_t(p_i)$ ($i=0,\dots,n$) satisfy
$m_0>\cdots>m_n$, the series $g$ lies in the ideal of
$B[[x_0,\dots,x_r]]$ that is generated by the monomials of
$\sigma_\ulm(x_0,\dots,x_r,1,\dots,1)$, where $\ulm=(m_0,\dots,m_n)$.
\end{itemize}
\end{lem}

\begin{proof}
The factor $g$ is unique since $x_i-x_j$ is not a zero divisor in the
power series ring, for $i<j$. Assertion (a) is clear since the
determinant vanishes after substitution $x_j:=x_i$.
Using the Leibniz formula, it suffices for (b) to consider the case
where each $p_i$ is a monomial $p_i=t^{a_i}$ with $a_i\ge m_i$.
So assume $p_i=t^{a_i}$ with $a_i\ge m_i$ for all~$i$. Expanding the
determinant by its lower $n-r$ rows, we see that $\det(M)$ is a
$B$-linear combination (usually infinite) of
$(r+1)\times(r+1)$-determinants
\begin{equation}\label{subdetindetlem}%
\det\begin{pmatrix}x_0^{b_0}&\cdots&x_0^{b_r}\\\vdots&&\vdots\\
x_r^{b_0}&\cdots&x_r^{b_r}\end{pmatrix}
\end{equation}
where $\ulb=(b_0,\dots,b_r)$ is a subsequence of
$\ula=(a_0,\dots,a_n)$. Say $b_i=a_{\nu(i)}$ with
$0\le\nu(0)<\cdots<\nu(r)\le n$,
let $\ulm'=(m_{\nu(0)},\dots,m_{\nu(r)})$ and let
$\ulb'=(b'_0,\dots,b'_r)$ be the descending permutation of $\ulb$.
We may assume that $b_i\ne b_j$ for $i\ne j$,
then \eqref{subdetindetlem} is equal to
$\pm\sigma_{\ulb'}(x_0,\dots,x_r)\cdot
\prod_{0\le i<j\le r}(x_i-x_j)$. It is easily checked that
$b'_i\ge m'_i=m_{\nu(i)}$ holds for $i=0,\dots,r$.
Therefore every monomial in $\sigma_{\ulb'}(x_0,\dots,x_r)$ is
divisible by some monomial in $\sigma_{\ulm'}(x_0,\dots,x_r)$, by
Lemma \ref{schurlem1a}, and hence by a monomial in
$\sigma_\ulm(x_0,\dots,x_r,1,\dots,1)$ (Lemma \ref{schurlem1}). So
the proof is complete.
\end{proof}


\section{Extreme rays}\label{extprep}%

\begin{lab}\label{asspts}%
Let $C$ be an affine algebraic curve over $\R$ that is irreducible
and non-singular. As before we write $A=\R[C]$ for the affine
coordinate ring of $C$. We assume that $C(\R)\ne\emptyset$, so
$C(\R)$ is topologically a union of finitely many loops ($1$-spheres)
and copies of the affine line~$\R$. Let $V$ be a linear subspace of
$A$ of dimension $n+1<\infty$, with basis $p_0,\dots,p_n$. For the
following fix a compact \sa\ subset $S\ne\emptyset$ of $C(\R)$
without isolated points.
We are going to study the convex cone $P=P_{V,S}$ in $V$ consisting
of all $f\in V$ that are non-negative on~$S$. Clearly, the cone $P$
is closed and pointed.
\end{lab}

\begin{dfn}\label{dfnvf}%
Given
an element $f\in V$, put
$$V_f\>=\>\{g\in V\colon\all\xi\in S\ \ord_\xi(g)\ge\ord_\xi(f)\}.$$
So $V_f$ is the linear space consisting of all $g\in V$ with at least
the same zeros in $S$ as $f$, taking multiplicities into account.
\end{dfn}

\begin{prop}\label{extstrahl}%
Let $S\ne\emptyset$ be a compact \sa\ set in $C(\R)$ without isolated
points, and let $0\ne f\in P=P_{V,S}$.
\begin{itemize}
\item[(a)]
$V_f=F_f-F_f$, the linear span of the supporting face $F_f$ of $f$
in $P$.
\item[(b)]
In particular, if $f$ spans an extreme ray of $P$ then $V_f=\R f$.
\end{itemize}
\end{prop}

\begin{proof}
It suffices to prove (a). For this let $B$ be the ring of all
rational functions on $C$ that have no pole in any point of~$S$.
Considering $A$ as a subring of $B$ we have $V_f=V\cap fB$.
The supporting face of $f$ in $P$ is $F_f=\{p\in P\colon f-p\in P\}$.
If $p,\,q\in P$ satisfy $f=p+q$, then clearly $p,\,q\in V_f$, showing
that $F_f\subset V_f$.
To prove that $V_f$ is spanned by $F_f$, we show that there exists an
open neighborhood $\itOmega$ of the origin in $V_f$ such that
$f+\itOmega\subset F_f$.
Let $W=\{g\in B\colon fg\in V_f\}$, then
the linear map $W\to V_f$, $g\mapsto fg$ is an isomomorphism.
Let $||\cdot||_S$ denote the supremum norm on $S$.
Then $\itOmega'=\{g\in W\colon||g||_S<1\}$ is an open neighborhood of
the origin in $W$, and so $f\itOmega'$ is a neighborhood of $0$ in
$V_f$. Since for every $g\in\itOmega'$ we have $f(1\pm g)\in P$, it
follows that $f(1+\itOmega')\subset F_f$, completing the proof.
\end{proof}

As a consequence, note that every extreme ray of $P$ is determined
(inside $V$) by its zeros in $S$, listed with multiplicities. Since
prescribing one zero means one linear condition on elements of $V$,
we also conclude:

\begin{cor}\label{kor2extstrahl}%
If $f\in P$ spans an extreme ray of $P$, then $f$ has at least $n$
zeros in $S$, counted with multiplicities.
\qed
\end{cor}

We remark that the results of this section remain true
\emph{verbatim} if the ground field is an arbitrary real closed field
$R$, instead of the field of real numbers. Of course one has to
replace ``convex'' by ``$R$-convex'', the corresponding notion
over~$R$ (cf.\ \cite[Cor.~1.6.18]{Sch}).

If $f$ spans an extreme ray of $P$, we'll later use Proposition
\ref{extstrahl} to obtain a determinantal identity for $f$, at least
in the case where $f$ has exactly $n$ zeros in~$S$. See Proposition
\ref{simplema} and Section~\ref{sectpfmainthm}.


\section{Outline of the proof}\label{outline}%

We now give a very coarse outline for the proof of Theorem
\ref{mainthm2} (following the reductions made in \ref{reduct2sosx})
and for the rest of this paper.
Given the point $\xi\in C(\R)$, the linear system $V$ in $A=\R[C]$
has a basis $p_0,\dots,p_n$ such that the sequence
$\ulm=(m_0,\dots,m_n)$ of vanishing orders $m_i=\ord_\xi(p_i)$ is
strictly decreasing. If $S$ is a closed interval on the positive
side of $\xi$, and if $S$ is sufficiently small, a general extremal
member $f$ of $(P_{V,S})_R$ (for $R\supset\R$ any real closed field)
will have precisely $n$ zeros in $S_R$, and these zeros determine
$f$ up to scaling. If we fix the number $r$ of different zeros
$\xi_1,\dots,\xi_r$ and the tuple $\ulb=(b_1,\dots,b_r)$ of their
(even) multiplicities, there exists a Nash function (analytic and
algebraic) $G(x_0,x_1,\dots,x_r)$, defined on a neighborhood of
$(\xi,\dots,\xi)$ in $C(\R)^{r+1}$, such that $f$ is given as
\begin{equation}\label{prodecomp}%
f(x)\>=\>G(x,\xi_1,\dots,\xi_r)\cdot\prod_{j=1}^r(t(x)-t(\xi_j))
^{b_j}
\end{equation}
in a neighborhood of $x=\xi$.
Writing $G$ as a (convergent) power series in $(t_0,\dots,t_r)$, the
lowest degree terms of $G$ are given by a Schur polynomial
\begin{equation}\label{cofactorschur}%
\sigma_\ulm(t_0,t_1,\dots,t_1,\dots,t_r,\dots,t_r),
\end{equation}
up to a scalar factor. ``Morally'', the tensor evaluation
$f^\otimes(\eta)$ should be the image of
\begin{equation}\label{morally}%
G(\eta,\xi_1,\dots,\xi_r)\cdot\prod_{j=1}^r(t(\eta)-t(\xi_j))^{b_j}
\end{equation}
under a ring homomorphism $\R\langle x_0,\dots,x_r\rangle\to
R\otimes R$ that maps $x_0$ to $t(\eta)\otimes1$ and $x_j$ to
$1\otimes t(\xi_j)$ for $j=1,\dots,r$. (Here $\R\langle\cdots\rangle$
denotes the ring of algebraic (Nash) power series.) The desired
$\sosx$-invariant $k=1+\lfloor\frac n2\rfloor$ should then come out
from \eqref{morally} correctly, after some further discussion.

When $C$ is a monomial curve, parametrized by a tuple
$(x^{m_0},\dots,x^{m_n})$ of monomials, this approach works and was
essentially carried out in \cite{avsch}. In fact, the cofactor $G$
agrees with the Schur polynomial \eqref{cofactorschur} in this case.
So for monomial curves, there is no need to use power series.

In the general case, however, there is one major problem, apart from
a number of detail questions that have been suppressed: A ring
homomorphism $\R\langle x_0,\dots,x_r\rangle\to R\otimes R$ as above
\emph{simply does not exist}. The approach via (Nash) power series is
too coarse. Instead it is necessary to work in an $(n+1)$-fold tensor
product $A^{\otimes(n+1)}$ over $\R$, to have the algebraic
dependence between the arguments built into the setup.
The backdrop is that there won't be a product decomposition
\eqref{prodecomp} any more. The vanishing ideal in $A^{\otimes(n+1)}$
of the generalized diagonal is not principal unless the curve is
rational, and so one has to work with suitable ideal generators and
with a corresponding decomposition. Arriving at a proper substitute
for a product decomposition \eqref{morally} is a major technical
step that will be carried out in the next two sections. The actual
proof of the main theorem will then be given in
Section~\ref{sectpfmainthm}.


\section{Representing the determinant, I}\label{repdet1}%

In this section and the next we are working on a non-singular affine
algebraic curve. Essentially, the base field won't play a role,
so we just assume that $k$ is a field of characteristic zero. Let $C$
be an integral affine curve over $k$ that is non-singular, with
affine coordinate ring $A=k[C]$. If $\xi\in C(k)$ is a $k$-rational
point, the maximal ideal of $A$ corresponding to $\xi$ is denoted
$\m_\xi$.

\begin{lab}\label{ikermu}%
We consider $A\otimes A=A\otimes_kA$ as an $A$-algebra via the
second tensor component, i.e.\ write $a(b\otimes c):=a\otimes(bc)$
for $a,\,b,\,c\in A$.
Let $\mu\colon A\otimes A\to A$ be the product map, let
$I=\ker(\mu)$. For $p\in A$ let $\delta(p):=p\otimes1-1\otimes p$,
an element of~$I$. The $A$-module $I/I^2$ is naturally isomorphic to
the module $\Omega_{A/k}$ of K\"ahler differentials of $A$ over $k$,
via $\Omega_{A/k}\isoto I/I^2$, $dp\mapsto\delta(p)+I^2$ ($p\in A$).
Note that $I$ is the vanishing ideal of the diagonal
$\Delta_C\subset C\times C$, and is an invertible prime ideal of
$A\otimes A$.
\end{lab}

\begin{lem}
The ideal $I$ of $A\otimes A$ is generated by the elements
$\delta(p)$ ($p\in A$).
\end{lem}

\begin{proof}
If $\alpha=\sum_ia_i\otimes b_i$ lies in $I$ then
$\alpha=\sum_ib_i\delta(a_i)$.
This proves the lemma (and shows that $I$ is generated by the
$\delta(p)$ even as an $A$-submodule of~$A\otimes A$).
\end{proof}

\begin{lem}\label{ikermugen}%
For $g\in A$ and $\eta\in C(k)$, the following are equivalent:
\begin{itemize}
\item[(i)]
$\ord_\eta(g-g(\eta))=1$;
\item[(ii)]
$dg$ generates the sheaf $\itOmega_{C/k}$ on $C$, locally at
$\eta$;
\item[(iii)]
$\delta(g)$ generates the ideal $I$ of $A\otimes A$, locally at
$(\eta,\eta)$.
\end{itemize}
\end{lem}

\begin{proof}
(i) and (ii) are equivalent since the stalk of the sheaf
$\itOmega_{C/k}$ at $\eta$ is $\m_\eta/\m_\eta^2$,
and since $dg$ corresponds to the coset of $g-g(\eta)$ in
$\m_\eta/\m_\eta^2$ under this isomorphism.
Also (iii) $\To$ (ii) is clear since $\Omega_{A/k}\cong I/I^2$
via the isomorphism $dg\mapsto\delta(g)+I^2$.
Conversely assume (ii) that $\itOmega_{C/k}$ is generated by $dg$
locally at~$\eta$. Replacing $C$ by a suitable neighborhood of
$\eta$, we may assume that $dg$ generates the $A$-module
$\Omega_{A/k}$. Let $M$ be the maximal ideal of $A\otimes A$
corresponding to $(\eta,\eta)$. Then $I\subset M$, and
$\delta(p)\in\frac{dp}{dg}\delta(g)+I^2\subset\idl{\delta(g)}+MI$
holds for every $p\in A$. Since $I$ is generated by the elements
$\delta(p)$ ($p\in A$), it follows that $I=\idl{\delta(g)}+MI$.
So the Nakayama lemma implies
$I(A\otimes A)_M=\delta(g)(A\otimes A)_M$.
\end{proof}

\begin{lab}\label{lotofdef}%
Now we need quite a bit of notation. Fix an integer $n\ge1$ and put
$A_n=A^{\otimes(n+1)}=A\otimes\cdots\otimes A$, the $(n+1)$-fold
tensor product over~$k$. The tensor components of $A_n$ will be
labelled with $i=0,\dots,n$. For $0\le i\le n$ let
$\varphi_i\colon A\to A_n$ be the $i$-th canonical embedding, i.e.\
$\varphi_i(a)=1\otimes\cdots\otimes a\otimes\cdots\otimes1$ with $a$
at position~$i$. For $0\le i<j\le n$ and $a\in A$ write
$\delta_{ij}(a)=\varphi_i(a)-\varphi_j(a)$.
Moreover let $\mu_{ij}\colon A_n\to A_n$ be the homomorphism that
multiplies the $i$-th and the $j$-th tensor component and puts the
result at position $i$, while putting $1$ at position~$j$. In other
words,
$$\mu_{ij}(a_0\otimes\cdots\otimes a_n)\>=\>
b_0\otimes\cdots\otimes b_n$$
where $b_i=a_ia_j$, $b_j=1$ and $b_\nu=a_\nu$ for
$\nu\in\{0,\dots,n\}\setminus\{i,j\}$.
The ideal $I_{ij}=\ker(\mu_{ij})$ is an invertible prime ideal of
$A_n$.
Let the ideal $\calI$ in $A_n$ be defined by
\begin{equation}\label{dfncali}%
\calI\>:=\>\bigcap_{0\le i<j\le n}I_{ij}\>=\>
\prod_{0\le i<j\le n}I_{ij}
\end{equation}
Intersection and ideal product coincide since the $I_{ij}$ are
pairwise different invertible prime ideals, and since the local rings
of $A_n$ are factorial.
In particular we see that $\calI$ is again an invertible ideal
of $A_n$, which corresponds to the generalized diagonal
$\bigcup_{i<j}\{(x_0,\dots,x_n)\colon x_i=x_j\}$ in $C^{n+1}$.
\end{lab}

\begin{lem}\label{neuansatz}%
For $w\in A$ let $\delta_w\in A_n$ be defined by
$$\delta_w\>=\>\prod_{0\le i<j\le n}\delta_{ij}(w).$$
Fixing an arbitrary point $\xi\in C(k)$, the ideal $\calI$ of $A_n$
is generated by all elements $\delta_w$ where $w\in A$ satisfies
$\ord_\xi(w)=1$.
\end{lem}

\begin{proof}
It suffices to prove the lemma locally, in a Zariski neighborhood of
$\uleta=(\eta_0,\dots,\eta_n)$ for any given tuple
$\uleta\in C(k)^{n+1}$.
Fix $\uleta$ and let
$(p_{ij})_{0\le i<j\le n}$ be a family of elements of $A$. Then
$\calI$ is generated by the product
$\prod_{0\le i<j\le n}\delta_{ij}(p_{ij})$ locally at $\uleta$,
provided that the following two conditions hold for each pair $i<j$
of indices:
\begin{itemize}
\item[(1)]
$p_{ij}(\eta_i)\ne p_{ij}(\eta_j)$ if $\eta_i\ne\eta_j$;
\item[(2)]
$\ord_{\eta_i}(p_{ij}-p_{ij}(\eta_i))=1$ if $\eta_i=\eta_j$.
\end{itemize}
This follows by applying Lemma \ref{ikermugen} to $I_{ij}$, for each
pair $i<j$.
To prove the lemma, it therefore suffices to show: Given a finite
number $\xi_0,\dots,\xi_r$ of pairwise different points in $C(k)$,
there exists $w\in A$ with $w(\xi_i)\ne w(\xi_j)$ for $i\ne j$,
with $\ord_{\xi_i}(w-w(\xi_i))=1$ for every $i$, and with
$\ord_\xi(w)=1$. Using the Chinese remainder theorem, it is clear
that this condition is satisfied.
\end{proof}

\begin{lab}\label{allgsetup}%
From now on we are going to impose some additional assumptions on the
curve $C$. Once and for all, we fix a $k$-point $\xi\in C(k)$ on
$C$ and let $\m_\xi$ denote the maximal ideal of $A$ corresponding
to~$\xi$. Moreover we assume that there is an element $t\in A$ such
that
\begin{enumerate}
\item[(A1)]
$\m_\xi=At$ (the principal ideal generated by $t$ in $A$),
\item[(A2)]
the $A$-module $\Omega_{A/k}$ is (freely) generated by $dt$,
\end{enumerate}
compare \ref{addlasspts2}. We'll also fix $t$ for the rest of the
section. If $p\in A$, let $p'=\frac{dp}{dt}\in A$ be defined by
$dp=p'\,dt$. This is well-defined by (A2).
\end{lab}

\begin{lab}\label{nmlzation}%
Since the ring $A_n$ is noetherian, Lemma \ref{neuansatz} implies
that there exists a finite subset $\itLambda$ of $A$ such that the
ideal $\calI$ is generated by the $\delta_w$ ($w\in\itLambda$), and
such that $\ord_\xi(w-t)>1$ (in particular, $\ord_\xi(w)=1$ and
$w'(\xi)=1$) holds for every $w\in\itLambda$.

\end{lab}

We need a few simple observations. The following lemma is easily
verified:

\begin{lem}\label{valtuple}%
Let $V$ be a linear subspace of $A$ with $\dim(V)=n+1$. Then the set
$\{\ord_\xi(p)\colon0\ne p\in V\}$ has cardinality $n+1$.
\qed
\end{lem}

Given a subspace $V\subset A$ of dimension $n+1$ we write
$\ulm_\xi(V)=(m_0,\dots,m_n)$ if $m_0>\cdots>m_n$ and there
exist elements $p_i\in V$ with $\ord_\xi(p_i)=m_i$ ($i=0,\dots,n$).
Note that any such tuple $\ulp=(p_0,\dots,p_n)$ is a $k$-basis
of~$V$. Given tuples $\ula=(a_0,\dots,a_n)$, $\ulb=(b_0,\dots,b_n)$
of integers, write $\ula\ge\ulb$ if $a_i\ge b_i$ for $i=0,\dots,n$,
and write $\ula>\ulb$ if $\ula\ge\ulb$ and $\ula\ne\ulb$.

\begin{lem}\label{ordxisubspace}%
Let $q_0,\dots,q_n$ be a $k$-basis of $V$, write $a_i=\ord_\xi(q_i)$
($i=0,\dots,n$) and assume $a_0\ge\cdots\ge a_n$. Then
$\ulm_\xi(V)\ge(a_0,\dots,a_n)$.
\end{lem}

\begin{proof}
If the $a_i$ are pairwise different there is nothing to show.
Otherwise let $i$ be the largest index with $a_i=a_{i+1}$. There is
(unique) $c\in k^*$ with $\ord_\xi(q_i+cq_{i+1})>\ord_\xi(q_i)$. Put
$a'_i=\ord_\xi(q_i+cq_{i+1})$ and $a'_\nu=a_\nu$ for $\nu\ne i$. If
$(\tilde a_0,\dots,\tilde a_n)$ is the weakly descending permutation
of $(a'_0,\dots,a'_n)$ then
$(\tilde a_0,\dots,\tilde a_n)\ge(a_0,\dots,a_n)$ holds.
So the lemma follows by descending induction on~$i$.
\end{proof}

\begin{lem}\label{ordxisubspace2}%
Let $\ulm_\xi(V)=(m_0,\dots,m_n)$, and let $\ulq=(q_0,\dots,q_n)$ be
a linearly independent sequence in $A$ such that
$\ord_\xi(q_i)\ge m_i$ for $i=0,\dots,n$, with strict inequality for
at least one index~$i$. Then $W=\spn(\ulq)$ satisfies
$\ulm_\xi(W)>\ulm_\xi(V)$.
\end{lem}

\begin{proof}
Let $a_i=\ord_\xi(q_i)$. The assertion is clear from Lemma
\ref{ordxisubspace} if $\ula=(a_0,\dots,a_n)$ is weakly decreasing.
Otherwise there are indices $i<j$ with $a_i<a_j$. Swap $a_i$ and
$a_j$ in $\ula$, then the new sequence $\ula'=(a'_0,\dots,a'_n)$
satisfies the hypothesis in the lemma as well.
After finitely many steps we have therefore reduced to Lemma
\ref{ordxisubspace}.
\end{proof}

\begin{lab}\label{grundlageF}%
For the discussion to follow, fix a linear subspace $V\subset A$ of
dimension $n+1$. Write $\ulm=\ulm_\xi(V)=(m_0,\dots,m_n)$ and let
$\ulp=(p_0,\dots,p_n)$ be a basis of $V$ with $\ord_\xi(p_i)=m_i$
($i=0,\dots,n$). Consider the matrix
\begin{equation}\label{dfnmp}%
M\>=\>M(\ulp)\>=\>\bigl(\varphi_i(p_j)\bigr)_{0\le i,j\le n}
\end{equation}
of size $(n+1)\times(n+1)$ with entries in $A_n$, together with its
determinant
\begin{equation}\label{dfnfp}%
F\>=\>F(\ulp)\>=\>\det M(\ulp)\in A_n.
\end{equation}
Note that $\mu_{ij}(F)=0$ for any pair of indices $0\le i<j\le n$,
since rows $i$ and $j$ of the matrix $\mu_{ij}(M)$ coincide.
Therefore $F$ lies in the ideal $\calI=\bigcap_{i<j}I_{ij}$. Note
also that, up to a nonzero scalar factor in $k$, the determinant
$F=F(\ulp)$ depends only on $V$, and not on the choice of the
basis~$\ulp$.
\end{lab}

\begin{lab}\label{dfnjjm}%
For $0\le i\le n$ write $t_i=\varphi_i(t)$ where $t\in A$ is the
element fixed in \ref{allgsetup}. Let $J=\idl{t_0,\dots,t_n}$, the
ideal generated by $t_0,\dots,t_n$ in $A_n$. By hypothesis (A1), this
is the maximal ideal of $A_n$ corresponding to the point
$(\xi,\dots,\xi)$ on the diagonal. From \ref{schurpolnot} recall the
definition of the Schur polynomial $\sigma_\ulm(x_0,\dots,x_n)$. By
$J(\ulm)$ we denote the ideal in $A_n$ that is generated by those
monomials $\ult^\alpha=t_0^{\alpha_0}\cdots t_n^{\alpha_n}$ that
occur in $\sigma_\ulm(t_0,\dots,t_n)$ with a nonzero coefficient. Our
goal in this section is to arrive at a specific identity for
$F=\det(M)$ (Corollary \ref{cor2nachtrick2}). The first step is to
show:
\end{lab}

\begin{prop}\label{allejm}%
$F=F(\ulp)$ is contained in the ideal product $\calI\cdot J(\ulm)$
(taken in the ring $A_n$).
\end{prop}

\begin{proof}
It suffices to prove the proposition after extension of the base
field.
We therefore assume that $k$ is algebraically closed (this is just to
simplify language and notation).
It suffices to argue in the local ring of any given tuple
$\uleta=(\eta_0,\dots,\eta_n)\in C(k)^{n+1}$. Let
$\scrO_\uleta=\scrO_{C^{n+1},\uleta}$ be the local ring at $\uleta$
and let $\wh\scrO_\uleta$ be its completion. Since the ring extension
$\scrO_\uleta\subset\wh\scrO_\uleta$ is faithfully flat
\cite[\S8]{Mat}, it suffices to show
$F\in\calI J(\ulm)\wh\scrO_\uleta$.

After a suitable permutation of the components, we may assume that
$\eta_0=\cdots=\eta_r=\xi$ (the point on $C$ fixed in
\ref{allgsetup}) and $\eta_i\ne\xi$ for $i=r+1,\dots,n$. Locally at
$\uleta$, the ideal $\calI$ is generated by
$$\prod_{i<j\text{ with }\eta_i=\eta_j}\delta_{ij}(t),$$
according to Lemma \ref{ikermugen} (this uses hypothesis (A2) from
\ref{allgsetup}).
On the other hand, the element $t\in A$ is a unit in
$\scrO_{C,\xi'}$ for every $\xi'\ne\xi$ in $C(k)$, by hypothesis (A1)
in \ref{allgsetup}.
Locally at~$\uleta$, this implies that the ideal $J(\ulm)$ is
generated by all monomials $t_0^{e_0}\cdots t_r^{e_r}$ that occur in
$\sigma_\ulm(t_0,\dots,t_r,1,\dots,1)$.
Proposition \ref{allejm} therefore follows from a repeated
application of Lemma \ref{detlem2}: First apply the lemma to rows
$0,\dots,r$ of $M=M(\ulp)$. Thereby we extract the factor
$\prod_{0\le i<j\le r}\delta_{ij}(t)$ and get a cofactor that is
contained in $J(\ulm)$. Then apply \ref{detlem2} successively to the
groups of rows whose indices lie in $\{i\colon\eta_i=\omega\}$, for
each of the remaining components $\omega$ of $\uleta$. This completes
the proof of the lemma: If $i<j$ are indices with $i,\,j>r$, the
factor $\delta_{ij}(t)=t_i-t_j$ is not a zero divisor modulo
$J(\ulm)$ since $J(\ulm)$ is generated by monomials in
$t_0,\dots,t_r$. Therefore $\delta_{ij}(t)G\in J(\ulm)$ implies
$G\in J(\ulm)$ for $G\in\scrO$.
\end{proof}

In more explicit terms, Proposition \ref{allejm} states:

\begin{cor}\label{cor2allejm}%
Under the assumptions made in \ref{grundlageF}, there exists an identity
\begin{equation}\label{sumgede}%
F(\ulp)\>=\>\sum_{w\in\itLambda}g_w\delta_w
\end{equation}
in $A_n$ with $g_w\in J(\ulm)$ for every $w\in\itLambda$.
\qed
\end{cor}

\begin{example}
To illustrate the proof of Proposition \ref{allejm} with a concrete
example, let $n=4$ and $\ulm=(5,4,3,2,0)$. So $\ulp=(p_0,\dots,p_4)$
is a tuple in $A$ with $\ord_\xi(p_i)=m_i$ for $i=0,\dots,4$.
The Schur polynomial is
$$\sigma_\ulm(x)\>=\>x_0x_1x_2x_3+x_0x_1x_2x_4+x_0x_1x_3x_4+
x_0x_2x_3x_4+x_1x_2x_3x_4,$$
the fourth elementary symmetric polynomial in $x_0,\dots,x_4$. Let
$\omega\ne\xi$ be a point in $C(k)$. By way of example, let us show
that $F=F(\ulp)$ lies in $\calI J(\ulm)$ locally at
$\uleta=(\xi,\xi,\xi,\omega,\omega)\in C(k)^5$. Let $u=t-t(\omega)$,
a local parameter at $\omega$, let $\scrO=\scrO_{C^5,\uleta}$ and let
$\wh\scrO$ be the completion of $\scrO$. Then $\wh\scrO$ is
identified with the formal power series ring
$k[[t_0,t_1,t_2,u_3,u_4]]$. In $\scrO$, the ideal $J(\ulm)$ is
generated by $t_0t_1,t_0t_2,t_1t_2$ since $t_3,t_4$ are units at
$\uleta$. Since $\calI$ is generated by
$\delta_{01}(t)\delta_{02}(t)\delta_{12}(t)\delta_{34}(t)$ at
$\uleta$, we have to show
$$F\>\in\>(t_0-t_1)(t_0-t_2)(t_1-t_2)(t_3-t_4)\cdot
\idl{t_0,t_1,t_2}$$
in $\wh\scrO$.
By Lemma \ref{detlem2}(b), $F$ is divisible by
$\delta_{01}(t)\delta_{02}(t)\delta_{12}(t)$, and the cofactor $G$ is
a power series, each of whose monomials is divisible by a monomial of
$\sigma_\ulm(t_0,t_1,t_2,1,1)$.
In other words $G\in J(\ulm)$. On the other hand, $F$ is
divisible by $\delta_{34}(u)=t_3-t_4$ as well (again by
\ref{detlem2}), say $F=(t_3-t_4)H$. And
$G\in J(\ulm)=\idl{t_0t_1,t_0t_2,t_1t_2}$ implies $H\in J(\ulm)$ as
well,
which completes what we wanted to prove.
\end{example}

\begin{lab}\label{2ndstep}%
In a second step we are going to refine identity \eqref{sumgede}.
As before, consider a $k$-basis $\ulp=(p_0,\dots,p_n)$ of $V$ with
$\ord_\xi(p_i)=m_i$, where $(m_0,\dots,m_n)=\ulm_\xi(V)=\ulm$. To get
rid of inessential constants we scale the $p_i$ in such a way that
$\ord_\xi(p_i-t^{m_i})>m_i$ for $i=0,\dots,n$.
By the normalization made in \ref{nmlzation}, this also implies
$\ord_\xi(p_i-w^{m_i})>m_i$ ($i=0,\dots,n$) for every
$w\in\itLambda$.
\end{lab}

\begin{lem}\label{jmjmm}%
Let $\ulq=(q_0,\dots,q_n)$ be a second tuple in $A$, and assume that
$\ord_\xi(q_i)\ge m_i$ holds for all $i$, with strict inequality for
at least one index~$i$. Then $F(\ulq)\in JJ(\ulm)\cdot\calI$.
\end{lem}

\begin{proof}
If $\ulq$ is $k$-linearly dependent there is nothing to show.
Otherwise let $W=\spn(\ulq)$ and put $\ula=\ulm_\xi(W)$. By Lemma
\ref{ordxisubspace2} we have $\ula>\ulm$ and therefore
$J(\ula)\subset JJ(\ulm)$ (cf.\ Lemma \ref{schurlem1a}). Applying
Proposition \ref{allejm} to $W$ and the sequence $\ulq$ gives
$F(\ulq)\in J(\ula)\calI\subset JJ(\ulm)\calI$, thereby proving the
assertion.
\end{proof}

\begin{lem}\label{vorhbemerk}%
For each $w\in\itLambda$ we may refine \eqref{sumgede} to an
identity
$$F(\ulp)\>=\>\sigma_\ulm(t_0,\dots,t_n)\delta_w+\sum_{v\in\itLambda}
g_v\delta_v,$$
where the $g_v$ are elements lying in $JJ(\ulm)$ for every
$v\in\itLambda$.
\end{lem}

\begin{proof}
Fix $w\in\itLambda$. Then $p_i=w^{m_i}+q_i$ holds with
$\ord_\xi(q_i)>m_i$ ($i=0,\dots,n$), by the normalization made in
\ref{2ndstep}. By multilinearity of the determinant, this implies
$$F(\ulp)\>=\>F(w^{m_0},\dots,w^{m_n})+\sum_\nu
F(\tilde q_{\nu0},\dots,\tilde q_{\nu n})$$
(sum over $2^{n+1}-1$ indices~$\nu$) where
$\tilde q_{\nu i}\in\bigl\{w^{m_i},\,q_i\bigr\}$, and where for each
index $\nu$ there exists at least one index $i$ with
$\tilde q_{\nu i}=q_i$. Lemma \ref{jmjmm} shows that
$F(\tilde q_{\nu0},\dots,\tilde q_{\nu n})\in JJ(\ulm)\calI$ for
each~$\nu$. This proves the lemma since
$$F(w_0^{m_0},\dots,w_n^{m_n})\>=\>\sigma_\ulm(w_0,\dots,w_n)\cdot
\delta_0$$
(bialternant formula \eqref{bialtfml}) and since
$$\sigma_\ulm(w_0,\dots,w_n)\>\equiv\>\sigma_\ulm(t_0,\dots,t_n)\
(\text{mod }JJ(\ulm))$$
by the assumption made in \ref{nmlzation}.
\end{proof}

\begin{prop}\label{neutrick}%
There is an identity $F(\ulp)=\sum_{w\in\itLambda}g_w\delta_w$ such
that $g_w$ lies in
$\frac1{|\itLambda|}\sigma_\ulm(t_0,\dots,t_n)+JJ(\ulm)$ for every
$w\in\itLambda$.
\end{prop}

\begin{proof}
For each $w\in\itLambda$ fix an identity
$$F(\ulp)\>=\>\sigma_\ulm(t_0,\dots,t_n)\delta_w+\sum_{v\in\itLambda}
g_{wv}\delta_v$$
with $g_{wv}\in JJ(\ulm)$ for all $v\in\itLambda$, see Lemma
\ref{vorhbemerk}. Adding all these identities and dividing by
$|\itLambda|$ gives an identity as asserted.
\end{proof}

Alternatively we may phrase Proposition \ref{neutrick} as follows:

\begin{cor}\label{cor2nachtrick2}%
There is an identity
\begin{equation}\label{cor2nachtrick2eq}%
F(\ulp)\>=\>\frac1{|\itLambda|}\sum_{w\in\itLambda}\sum_\alpha
t_0^{\alpha_0}\cdots t_n^{\alpha_n}(1+g_{w,\alpha})\cdot\delta_w
\end{equation}
where the inner sum is taken over all monomials $x^\alpha$ occuring
in $\sigma_\ulm(x)$, and where $g_{w,\alpha}\in J$ for all $w$
and~$\alpha$.
\qed
\end{cor}


\section{Representing the determinant, II}\label{repdet2}%

We keep the hypotheses made in \ref{allgsetup}. So $C$ is an
irreducible and non-singular affine curve over a field $k$ of
characteristic~$0$, and $\xi\in C(k)$ is a fixed point. The element
$t\in A=k[C]$ is a local uniformizer at $\xi$, and the $A$-module
$\Omega_{C/k}$ is freely generated by~$dt$. For $p\in A$ we keep
the notation $p'=\frac{dp}{dt}$, and write
$p^{(i)}=\frac d{dt}p^{(i-1)}$ ($i\ge1$) for the iterated
$\frac d{dt}$-derivatives.

Recall that $A\otimes A$ is regarded as an $A$-module via the
second tensor component, and that $I\subset A\otimes A$ denotes the
kernel of the product map $\mu\colon A\otimes A\to A$. The ideal $I$
is generated by the $\delta(p)=p\otimes1-1\otimes p$ ($p\in A$). For
$r\ge1$ let $I^r$ denote the $r$-th ideal power of $I$. Note that
$\alpha\beta\equiv\mu(\alpha)\beta$ (mod~$I^{r+1}$) holds for
$\alpha\in A\otimes A$ and $\beta\in I^r$.

\begin{lem}\label{taylorlem}%
For $f\in A$ and every integer $r\ge1$, the congruence
\begin{equation}\label{taylorcong}%
\delta(f)\>\equiv\>f'\cdot\delta(t)+\frac1{2!}f''\cdot\delta(t)^2+
\cdots+\frac1{r!}f^{(r)}\cdot\delta(t)^r\ {\rm(mod\ }I^{r+1})
\end{equation}
holds in $A\otimes A$.
\end{lem}

\begin{proof}
Let $\Omega=\Omega_{A/k}$. The isomorphism $\Omega\isoto I/I^2$,
$df\mapsto\delta(f)+I^2$ of $A$-modules induces
isomorphisms $\Sym^d_A(\Omega)\to I^d/I^{d+1}$ for all $d\ge0$, since
the curve $C$ is smooth (\cite{ega4} 17.12.4, 16.9.4).
Since $\Omega$ is freely generated by $dt$, this implies that there
exist unique elements $g_1,\dots,g_r\in A$ such that
\begin{equation}\label{taylorcong2}%
\delta(f)\>\equiv\>g_1\cdot\delta(t)+\cdots+g_r\cdot\delta(t)^r\
{\rm(mod\ }I^{r+1})
\end{equation}
holds in $A\otimes A$.
Given an arbitrary point $\eta\in C(k)$, consider the homomorphism
$\phi_\eta\colon A\otimes A\to A$ defined by
$\phi_\eta(p\otimes q)=p\cdot q(\eta)$. Then $\phi_\eta(I)=\m_\eta$,
so applying $\phi_\eta$ to \eqref{taylorcong2} gives
$$f-f(\eta)\>\equiv\>g_1(\eta)\cdot(t-t(\eta))+\cdots+g_r(\eta)\cdot
(t-t(\eta))^r\ {\rm(mod\ }\m_\eta^{r+1}).$$
Since $t-t(\eta)$ is a uniformizing element at $\eta$ (Lemma
\ref{ikermugen}) it follows that $g_i(\eta)=\frac1{i!}f^{(i)}(\eta)$
for $i=1,\dots,r$. This proves the lemma.
\end{proof}

\begin{lem}\label{dfndr}%
For every $r\ge1$ there exists a (unique) well-defined $A$-linear map
$D^{(r)}\colon I^r\to A$ that satisfies $D^{(r)}(I^{r+1})=0$ and sends
$\delta(p_1)\cdots\delta(p_r)$ to $p'_1\cdots p'_r$, for any
$p_1,\dots,p_r\in A$. This map satisfies
$D^{(r)}(\alpha\beta)=\mu(\alpha)D^{(r)}(\beta)$ for
$\alpha\in A\otimes A$ and $\beta\in I^r$.
\end{lem}

\begin{proof}
The graded $A$-algebra $\bigoplus_{r\ge0}I^r/I^{r+1}$ is the
symmetric algebra over the $A$-module $I/I^2$, see the previous
proof. Since $I/I^2$ is freely generated by $dt=\delta(t)+I^2$,
there is a unique homomorphism of $A$-algebras
$D\colon\bigoplus_{r\ge0}I^r/I^{r+1}\to A$ with $D(dt)=1$. Let
$D^{(r)}\colon I^r\to A$ be the induced $A$-linear map for
$r\ge0$, so $D^{(r)}(\beta)=D(\beta+I^{r+1})$ for $\beta\in I^r$.
We claim that $D^{(r)}$ has the above properties. For $a,\,b\in A$
one has $D^{(r)}((a\otimes b)\beta)=ab\cdot D^{(r)}(\beta)$ by
$A$-linearity, since $a\otimes b\equiv1\otimes ab$ (mod~$I$).
For $p\in A$ we have $dp=p'\,dt$ by definition of $p'$, which
means $\delta(p)\equiv p'\delta(t)$ (mod~$I$).
This implies $D^{(1)}(\delta(p))=p'$. Consequently, $D^{(r)}$
maps $\delta(p_1)\cdots\delta(p_r)$ to $p'_1\cdots p'_r$ ($r\ge1)$
since $D$ is a ring homomorphism.
\end{proof}

\begin{lab}
In the tensor product $A_n=A^{\otimes(n+1)}$ we may perform the
operations $D^{(r)}$ on any fixed pair $i<j$ of indices. Recall
that $I_{ij}$ denotes the kernel of $\mu_{ij}\colon A_n\to A_n$ and
is generated by the $\delta_{ij}(a)$ ($a\in A$) as an ideal in
$A_n$. For $0\le i<j\le n$ let
$$D^{(r)}_{ij}\colon I_{ij}^r\to A_n$$
be the additive map that satisfies
\begin{equation}\label{multruleDr}%
D^{(r)}_{ij}\bigl(\alpha\,\delta_{ij}(a_1)\cdots\delta_{ij}(a_r)
\bigr)\>=\>\mu_{ij}(\alpha)\cdot\varphi_i(a_1'\cdots a'_r)
\end{equation}
for $a_1,\dots,a_r\in A$ and $\alpha\in A_n$. The unique existence
of a map with these properties follows directly from
\ref{dfndr}. For any $p\in A$ and any $r\ge1$, note that
\begin{equation}\label{drtaylor}%
D_{ij}^{(r)}\biggl(\delta_{ij}(p)-\sum_{\nu=1}^{r-1}
\frac{\varphi_j(p^{(\nu)})}{\nu!}\delta_{ij}(t)^\nu\biggr)\>=\>
\frac1{r!}\varphi_i(p^{(r)}),
\end{equation}
by Lemma \ref{taylorlem}.
\end{lab}

\begin{lab}\label{taylorproc}%
As in Section \ref{repdet1}, let $\ulp=(p_0,\dots,p_n)$ be a fixed
$(n+1)$-tuple of elements in $A$, and consider the matrix
$M=M(\ulp)=\bigl(\varphi_i(p_j)\bigr)_{0\le i,j\le n}$ over $A_n$,
see \ref{grundlageF}. If $k\ge0$, write
$\ulp^{(k)}=(p_0^{(k)},\dots,p_n^{(k)})$ for the tuple of $k$-th
$\frac d{dt}$-derivatives. Fix an integer $1\le b\le n+1$ and a block
$B$ of $b$ consecutive numbers in $\{0,\dots,n\}$;
for simplicity we take $B=\{0,\dots,b-1\}$. In the following we
describe a transformation of the matrix $M$ to a new matrix
$\calT_BM$, to which we'll refer as the \emph{Taylor process on
$B$}.

Start with the matrix $M$ and let $z_i(M)=\varphi_i(\ulp)$ denote
the $i$-th row of $M$ ($i=0,\dots,n$). Replacing $z_1(M)$ with
$z_1(M)-z_0(M)$, this new row lies entrywise in $I_{01}$,
so we can apply $D^{(1)}_{01}$ to it. Leaving the other rows
unchanged, this gives the new matrix $M'$ with rows
$$\varphi_0(\ulp),\,-\varphi_0(\ulp'),\,
\varphi_2(\ulp),\,\dots,\,\varphi_n(\ulp).$$
Since the determinant is multilinear with respect to the rows, the
determinant of $M'$ satisfies
$$\det(M')\>=\>D^{(1)}_{01}(\det M)$$
by \eqref{multruleDr}. Next replace $z_2(M')=\varphi_2(\ulp)$ with
$$z_2(M')-z_0(M')-\delta_{02}(t)z_1(M')\>=\>
\varphi_2(\ulp)-\varphi_0(\ulp)-(t_2-t_0)\varphi_0(\ulp').$$
Entrywise, this row lies in $I_{02}^2$ and is congruent to
$\frac12\varphi_0(\ulp'')\delta_{02}(t)^2$ modulo $I_{03}^3$, both by
Lemma \ref{taylorlem}.
Apply $D^{(2)}_{02}$ to this row, to get the new matrix $M''$ with
rows
$$\varphi_0(\ulp),\,-\varphi_0(\ulp'),\,\frac12\varphi_0(\ulp''),\,
\varphi_3(\ulp),\,\dots,\,\varphi_n(\ulp),$$
whose determinant satisfies
\begin{equation}\label{dettransfmd}%
\det(M'')\>=\>D^{(2)}_{02}(\det M')\>=\>D^{(2)}_{02}\comp
D^{(1)}_{01}(\det M).
\end{equation}
Keep proceeding in this way, successively working down the rows in
the block~$B$. After $b-1$ steps we have arrived at a matrix
$\calT_BM=M^{(b-1)}$ with rows
\begin{equation}\label{taylorzeilen}%
\varphi_0(\ulp),\,-\varphi_0(\ulp'),\,\dots,\,
\frac{(-1)^{b-1}}{(b-1)!}\varphi_0(\ulp^{(b-1)}),\,\varphi_b(\ulp),
\,\dots,\,\varphi_n(\ulp)
\end{equation}
whose determinant satisfies
\begin{equation}\label{taylprocdet}%
\det(\calT_BM)\>=\>D^{(b-1)}_{0,b-1}\comp\cdots\comp D^{(1)}_{01}
(\det M).
\end{equation}
\end{lab}

\begin{lab}\label{operatoract}%
Fix an identity $\det(M)=\sum_{w\in\itLambda}g_w\delta_w$ for the
determinant of $M$, with $g_w\in J(\ulm)$ (Corollary
\ref{cor2allejm}). How does the Taylor process just described act on
the terms of this identity? Write
$\scrD_B:=D_{0b}^{(b)}\comp\cdots\comp D_{01}^{(1)}$ for the
composite operator (it is defined on $\bigcap_{j=1}^bI_{0j}$), and
let $\bfmu_B=\mu_{0b}\comp\cdots\comp\mu_{01}\colon A_n\to A_n$
denote the ring homomorphism
$$\bfmu_B\bigl(a_0\otimes\cdots\otimes a_n\bigr)\>=\>
(a_0\cdots a_{b-1})\otimes\underbrace{1\otimes\cdots\otimes1}
_{b-1\>\text{times}}\otimes\>a_b\otimes\cdots\otimes a_n.$$
Then $\scrD_B(g\delta_w)=\bfmu_B(g)\scrD_B(\delta_w)$ for
$g\in A_n$ and $w\in A$, where
\begin{equation}\label{scrDdeltaw}%
\scrD_B(\delta_w)\>=\>\varphi_0(w')^{b(b-1)/2}\cdot
\prod_{b\le j\le n}\delta_{0j}(w)^b\cdot
\prod_{b\le i<j\le n}\delta_{ij}(w).
\end{equation}
Note that the entries of the matrix $\calT_BM=M^{(b-1)}$ (and its
determinant) lie in the subring
$$A\otimes\underbrace{k\otimes\cdots\otimes k}_
{b-1\>\text{times}}\otimes
\underbrace{A\otimes\cdots\otimes A}_{n-b+1\>\text{times}}$$
of $A_n$.
\end{lab}

\begin{lab}\label{multipletaylor}%
As a final step we apply the Taylor process \ref{taylorproc} to
several disjoint blocks of rows of the matrix $M=M(\ulp)$, as
follows. We assume that the rows $0,\dots,n$ of the matrix have been
grouped into blocks $B_0,\dots,B_r$ of consecutive rows, with
$B_\nu$ consisting of $b_\nu$ many rows for $\nu=0,\dots,r$. More
formally, let $\ulb=(b_0,\dots,b_r)$ be a tuple of integers $b_i\ge1$
with $n+1=\sum_{i=0}^rb_i$, and let $B_0=\{0,\dots,b_0-1\}$,
$B_1=\{b_0,\dots,b_0+b_1-1\}$ etc, up to $B_r=\{n-b_r+1,\dots,n\}$.
For each of the blocks $B_0,\dots,B_r$ perform the Taylor process
\ref{taylorproc} on the rows of this block. In the end we have
arrived at a matrix $T=T_\ulb(\ulp)$ whose coefficients lie in the
subring
$$\bigl(A\otimes\underbrace{k\otimes\cdots\otimes k}_
{(b_0-1)\text{ times}}\bigr)\otimes\cdots\otimes
\bigl(A\otimes\underbrace{k\otimes\cdots\otimes k}_
{(b_r-1)\text{ times}}\bigr)$$
of $A_n$.
Dropping the inessential tensor components we consider $T$ to have
coefficients in $A_r=A^{\otimes(r+1)}$. With this convention, the
rows of $T$ are
\begin{equation}\label{rowsofn}%
\varphi_0(\ulp),\,\dots,\,\varphi_0(\ulp^{(b_0-1)}),\,
\dots,\,\varphi_r(\ulp),\,\dots,\,\varphi_r(\ulp^{(b_r-1)})
\end{equation}
up to nonzero scalar factors in $k$, see \ref{taylorproc}. In order
to describe the effect of the multiple Taylor process on the
determinant, let $\bfmu=\bfmu_\ulb\colon A_n\to A_r$ denote the ring
homomorphism that puts the product of the tensor components in block
$B_\nu$ at position $\nu$, for $\nu=0,\dots,r$:
\begin{equation}\label{dfnbfmu}%
\bfmu(a_0\otimes\cdots\otimes a_n)\>=\>
\Bigl(\prod_{i\in B_0}a_i\Bigr)\otimes\cdots\otimes
\Bigl(\prod_{i\in B_r}a_i\Bigr).
\end{equation}
Writing $\scrD=\scrD_{B_r}\comp\cdots\comp\scrD_{B_0}$ we have
$\scrD(g\delta_w)=\bfmu(g)\scrD(\delta_w)$ for $g\in A_n$ and
$w\in A$, where
\begin{equation}\label{scrDdeltaw}%
\scrD(\delta_w)\>=\>\prod_{i=0}^r\varphi_i(w')^{b_i(b_i-1)/2}\cdot
\prod_{0\le i<j\le r}\delta_{ij}(w)^{b_ib_j}.
\end{equation}
If we start with an identity
\begin{equation}\label{2cor2trick2}%
\det(M)\>=\>\frac1{|\itLambda|}\sum_{w,\alpha}
t_0^{\alpha_0}\cdots t_n^{\alpha_n}(1+g_{w,\alpha})\cdot\delta_w
\end{equation}
as in Corollary \ref{cor2nachtrick2}, we therefore get
\begin{equation}\label{detscrTM}%
\det T_\ulb(\ulp)\>=\>\frac1{|\itLambda|}\sum_{w,\alpha}
t_0^{\alpha(B_0)}\cdots t_r^{\alpha(B_r)}\cdot
(1+\bfmu(g_{w,\alpha}))\cdot\scrD(\delta_w)
\end{equation}
with $\alpha(B_\nu)=\sum_{j\in B_\nu}\alpha_j$ for $\nu=0,\dots,r$.
\end{lab}

In summary, the previous discussion gives the following result:

\begin{thm}\label{summazing}%
Let $b_0,\dots,b_r\ge1$ with $\sum_{i=0}^rb_i=n+1$, and consider
the $(n+1)\times(n+1)$ matrix $T_\ulb(\ulp)$ over $A_r$ with rows
\begin{equation}\label{rowsofn2}%
\varphi_0(\ulp),\,\dots,\,\varphi_0(\ulp^{(b_0-1)}),\,
\dots,\,\varphi_r(\ulp),\,\dots,\,\varphi_r(\ulp^{(b_r-1)}).
\end{equation}
Its determinant $\det T_\ulb(\ulp)$ can be written
\begin{equation}\label{identsummazing}%
\sum_{w,\alpha}\biggl(\bfmu(t^\alpha)\cdot(1+h_{w,\alpha})\cdot
\prod_{0\le i\le r}\varphi_i(w')^{b_i(b_i-1)/2}\cdot
\prod_{0\le i<j\le r}\delta_{ij}(w)^{b_ib_j}\biggr),
\end{equation}
up to a scaling factor in $k^*$. Here the $h_{w,\alpha}$ lie in
$\idl{t_0,\dots,t_r}$, and the sum is over $w\in\itLambda$ and those
multiindices $\alpha=(\alpha_0,\dots,\alpha_n)$ for which $x^\alpha$
occurs in the Schur polynomial $\sigma_\ulm(x_0,\dots,x_n)$.
\qed
\end{thm}

Given a $k$-algebra $\calA$ together with $k$-homomorphisms
$\psi_i\colon A\to\calA$ ($i=0,\dots,r$), we may specialize the
matrix $T_\ulb(\ulp)$ over $A_r$ to a matrix over $\calA$ via
$$\psi\colon A_r\to\calA,\quad a_0\otimes\cdots\otimes a_r\mapsto
\psi_0(a_0)\cdots\psi_r(a_r).$$
Several specializations of this sort will play a role in the sequel.
We begin with the following one:

\begin{prop}\label{simplema}%
Let $p_0,\dots,p_n$ be a basis of the vector space $V\subset A$, let
$\uleta=(\eta_1,\dots,\eta_r)$ be a tuple of $r$ pairwise different
points in $C(k)$, and let $\ulb=(b_1,\dots,b_r)$ be an $r$-tuple of
positive integers with $\sum_{i=1}^rb_i=n$. The following conditions
are equivalent:
\begin{itemize}
\item[(i)]
The matrix
\begin{equation}\label{bigmatbieq2}%
Z_\ulb(\ulp,\uleta)\ =\ \begin{pmatrix}p_0&\cdots&p_n\\
p_0(\eta_1)&\cdots&p_n(\eta_1)\\
\vdots&&\vdots\\
p_0^{(b_1-1)}(\eta_1)&\cdots&p_n^{(b_1-1)}(\eta_1)\\
\vdots&&\vdots\\
p_0(\eta_r)&\cdots&p_n(\eta_r)\\
\vdots&&\vdots\\
p_0^{(b_r-1)}(\eta_r)&\cdots&p_n^{(b_r-1)}(\eta_r)
\end{pmatrix}
\end{equation}
(of size $(n+1)\times(n+1)$ and with coefficients in $A$) has
non-zero determinant;
\item[(ii)]
the subspace $V_0:=\bigcap_{i=1}^r\{f\in V\colon\ord_{\eta_i}(f)
\ge b_i\}$ of $V$ has dimension one.
\end{itemize}
When (i) and (ii) hold, the regular function
$\det Z_\ulb(\ulp,\uleta)$ on $C$ (is non-zero and) lies
in~$V_0$.
\end{prop}

\begin{proof}
Put $Z:=Z_\ulb(\ulp,\uleta)$. Regardless of condition (i), the
subspace $V_0$ in (ii) has $\dim(V_0)\ge1$ since $V_0$ is described
by $n$ linear conditions on $f$. Let $Z_0$ be the matrix that is
obtained from $Z$ by deleting the top row. Then $Z_0$ is the matrix
of the linear map $\phi\colon V\to k^n$,
$$p\ \mapsto\ \Bigl(p(\eta_1),\dots,p^{(b_1-1)}(\eta_1),\dots,
p(\eta_r),\dots,p^{(b_r-1)}(\eta_r)\Bigr)$$
with respect to the basis $p_0,\dots,p_n$ of $V$. By definition,
$V_0$ is the kernel of $\phi$. The determinant $\det(Z)$ is an
element of $V$ which is non-zero if, and only if, at least one of
the $n\times n$-minors of $Z_0$ does not vanish.
This is equivalent to
$\phi$ being surjective, and hence also to condition~(ii). To see the
last assertion note that, for any $k\ge0$, the $k$-th derivative
$\frac{d^k}{dt^k}\det(Z)$ of $\det(Z)$ is obtained by replacing the
top row of $Z$ with $(p_0^{(k)},\dots,p_n^{(k)})$. Therefore, the
vanishing order of $\det(Z)$ at $\eta_i$ is at least $b_i$, for every
$i=1,\dots,r$.
\end{proof}

In the next section we are going to specialize the matrix
$T_\ulb(\ulp)$ in several other ways, to complete the proof of the
main theorem.


\section{Proof of the main theorem}\label{sectpfmainthm}%

\begin{lab}\label{need2prove}%
After the extensive preparations in the previous sections we can
now give the proof of the main theorem. Assume that an affine
algebraic curve $C$ over $\R$ is given that is non-singular and
irreducible, together with an $\R$-linear subspace $V$ of $A=\R[C]$
of dimension $n+1$. Moreover let $\xi\in C(\R)$ be a fixed real
point of $C$, together with a local orientation of the curve
$C(\R)$ at $\xi$. We assume that $t\in A$ represents the given
orientation and is such that $\Omega_{A/\R}$ is freely generated
by $dt$, and such that $\m_\xi=At$ (conditions (A1) and (A2) from
\ref{allgsetup}).
We need to prove the following (see \ref{reductevenmore},
\ref{reduct2sosx}):
\begin{itemize}
\item[$(*)$]
There exists a closed interval $S=[\xi,\xi']$ in $C(\R)$
on the positive side of $\xi$ for which the following is true: For
every real closed field $R\supset\R$, every element
$f\in V_R=V\otimes R$ with $f\ge0$ on $S_R\subset C(R)$ and every
$\xi_0\in S_R$, the tensor evaluation (see \ref{tenseval}) of $f$ at
$\xi_0$ satisfies
$$\sosx f^\otimes(\xi_0)\>\le\>1+\lfloor\frac n2\rfloor.$$
It suffices in fact to show this for every $f$ that spans an extreme
ray of the cone $(P_{V,S})_R=\{g\in V_R\colon g\ge0$ on~$S_R\}$.
\end{itemize}
See \ref{erklinterval} for the meaning of interval in this context.
\end{lab}

\begin{lab}\label{nobasepts}%
We may assume that $\xi$ is not a base point of $V$, i.e., that
there is $p\in V$ with $p(\xi)\ne0$. Indeed, let
$e=\min\{\ord_\xi(p)\colon p\in V\}$. Then $V=t^eV_1$ where
$V_1=\{t^{-e}p\colon p\in V\}$ is contained in $A$ since $\m_\xi=At$.
Now $\xi$ isn't a base point of $V_1$ any more, and multiplication by
$t^e$ is a linear isomorphism $V_1\to V$ that identifies $P_{V_1,S}$
with $P_{V,S}$ for every interval $S\subset C(\R)$ on which $t$ is
non-negative. So we may replace $V$ with $V_1$.
\end{lab}

\begin{lab}\label{findes}%
As in Sections \ref{repdet1} and \ref{repdet2} let
$\ulm=\ulm_\xi(V)$, and let $\ulp=(p_0,\dots,p_n)$ be an $\R$-basis
of $V$ with $\ord_\xi(p_i)=m_i$ ($i=0,\dots,n$). To find an interval
$S$ as in \ref{need2prove}, form the matrix $M=M(\ulp)$ over
$A_n=A^{\otimes(n+1)}$ as in \ref{grundlageF}. Let
$\sigma_\ulm(x_0,\dots,x_n)$ be the Schur polynomial associated with the
sequence $\ulm=(m_0,\dots,m_n)$, see \ref{schurpolnot}, and fix a
finite set $\itLambda\subset A$ of elements $w$ as in \ref{nmlzation}.
The ideal $\calI$ in $A_n$ is generated by the $\delta_w$
($w\in\itLambda$), and each $w\in\itLambda$ satisfies
$\ord_\xi(w)=1$ and $\frac{dw}{dt}(\xi)=1$. Then, according to
Corollary \ref{cor2nachtrick2}, there is an identity
\begin{equation}\label{Fidenty}%
\det M(\ulp)\>=\>\frac1{|\itLambda|}\sum_{w,\alpha}
t_0^{\alpha_0}\cdots t_n^{\alpha_n}(1+g_{w,\alpha})\cdot\delta_w
\end{equation}
in $A_n$ as in \eqref{cor2nachtrick2eq}. In particular, the element
$g_{w,\alpha}$ lies in $J=\idl{t_0,\dots,t_n}$ for every pair
$w,\alpha$. (Recall that the sum is over $w\in\itLambda$ and the
monomials $x^\alpha$ in $\sigma_\ulm(x)$.)

From now on fix one such identity \eqref{Fidenty}.
Considering the finitely many elements
$g_{w,\alpha}\in J\subset A_n$ as polynomial functions on
$C^{n+1}=C\times\cdots\times C$, they all vanish in the diagonal
point $(\xi,\dots,\xi)$. So there exists an open neighborhood $Q$ of
$\xi$ in $C(\R)$ with the property that
$1+g_{w,\alpha}(\eta_0,\dots,\eta_n)>0$ for all
$\eta_0,\dots,\eta_n\in Q$ and all $w,\alpha$. Choose the
non-degenerate interval $S=[\xi,\xi']$ in such a way that $S$
contains the positive side of $\xi$, and such that the following
hold:
\begin{itemize}
\item[(S0)]
$V$ has no base point in $S$,
\item[(S1)]
$S\subset Q$,
\item[(S2)]
$t\ge0$ on $S$,
\item[(S3)]
$w'=\frac{dw}{dt}\ge0$ on $S$ (and hence also $w\ge0$ on $S$), for
every $w\in\itLambda$.
\end{itemize}
Clearly it is possible to find such an interval $S$, given that
$\xi$ is not a base point of $V$ (see \ref{nobasepts}).
We claim that $(*)$ in \ref{need2prove} is satisfied for such $S$.
\end{lab}

\begin{lab}\label{fwithnzeros}%
Let $R\supset\R$ be a real closed field extension, let $f\in V_R$
with $f\ge0$ on $S_R$ and assume that $f$ spans an extreme ray of
$(P_{V,S})_R$. If $f(\xi)=0$ we may pass from $V$ to the linear
system $V_0=\{p\in V\colon p(\xi)=0\}$. Then
$\dim(V)=n<n+1$ since $\xi$ is not a base point of $V$, and by
induction we may assume this case already to be covered. Argueing
similarly in case $f(\xi')=0$, we may therefore assume that
$f(\xi)>0$ and $f(\xi')>0$. This implies that the number of zeros of
$f$ in $S_R$ is even, counting with multiplicities. By Corollary
\ref{kor2extstrahl}, this number is at least~$n$. We claim that it
is equal to $n$ (and therefore, that $n$ is even).
To see this, consider the following lemma:
\end{lab}

\begin{lem}\label{n+1eval}%
With $S$ as in \ref{findes}, assume that $\xi_0,\dots,\xi_r\in S_R$
are pairwise different and also different from $\xi$. Let integers
$b_0,\dots,b_r\ge1$ with $\sum_{i=0}^rb_i=n+1$ be given. Then the
determinant of the matrix
\begin{equation}\label{bigmatptsinS}%
\begin{pmatrix}
p_0(\xi_0)&\cdots&p_n(\xi_0)\\
\vdots&&\vdots\\
p_0^{(b_0-1)}(\xi_0)&\cdots&p_n^{(b_0-1)}(\xi_0)\\
\vdots&&\vdots\\
p_0(\xi_r)&\cdots&p_n(\xi_r)\\
\vdots&&\vdots\\
p_0^{(b_r-1)}(\xi_r)&\cdots&p_n^{(b_r-1)}(\xi_r)
\end{pmatrix}
\end{equation}
is a nonzero element of~$R$.
\end{lem}

\begin{proof}
On the interval $S$, each $w\in\itLambda$ is strictly increasing as a
function of $t$, by hypothesis (S3) in \eqref{Fidenty}. So we may
label the $\xi_i$ in such a way that
$$w(\xi_0)\>>\>w(\xi_1)\>>\>\cdots\>>\>w(\xi_r)$$
for each $w\in\itLambda$.
Consider the matrix $T_\ulb(\ulp)$ over $A_r$ from Theorem
\ref{summazing}, with parameters $\ulb=(b_0,\dots,b_r)$ as given
above. The matrix \eqref{bigmatptsinS} is the image of
$T_\ulb(\ulp)$ under the ring homomorphism
\begin{equation}\label{ringhomanr}%
A_r\to R,\quad a_0\otimes\cdots\otimes a_r\mapsto
\prod_{i=0}^ra_i(\xi_i)
\end{equation}
Hence the determinant of \eqref{bigmatptsinS} is the image of the sum
\eqref{identsummazing} under this ring homomorphism. But from
conditions (S0)--(S3) it follows that every factor in a typical
summand
$$\bfmu(t^\alpha)\cdot(1+h_{w,\alpha})\cdot\prod_i
\varphi_i(w')^{b_i(b_i-1)/2}\cdot\prod_{i<j}
\delta_{ij}(w)^{b_ib_j}$$
of \eqref{identsummazing} maps to a \emph{strictly positive} number
in $R$ under \eqref{ringhomanr}. This proves the lemma.
\end{proof}

\begin{prop}\label{genaunzeros}%
$f$ as in \ref{fwithnzeros} has precisely $n$ zeros ($\ne\xi,\,\xi'$)
in $S_R$, counting with multiplicities. In particular, $n$ is even.
\end{prop}

\begin{proof}
Assume to the contrary that $f$ has at least $n+1$ zeros in $S_R$
(different from $\xi,\,\xi'$). Then there are
$\xi_0,\dots,\xi_r\in S_R$, pairwise different, together with
integers $b_0,\dots,b_r\ge1$ such that $\ord_{\xi_i}(f)\ge b_i\ge1$
for $i=0,\dots,r$ and $\sum_{i=0}^rb_i=n+1$. By Lemma \ref{n+1eval},
the determinant of the matrix \eqref{bigmatptsinS} is nonzero. Choose
an arbitrary element $q\in A$ that is linearly independent from
$p_0,\dots,p_n$. Consider the extended tuple
$\tilde\ulp=(q,p_0,\dots,p_n)$ and form the matrix of size
$(n+2)\times(n+2)$ over $A$ with rows
$$\tilde\ulp,\ \tilde\ulp(\xi_0),\dots,\tilde\ulp^{(b_0-1)}(\xi_0),\
\dots,\ \tilde\ulp(\xi_r),\dots,\tilde\ulp^{(b_r-1)}(\xi_r)$$
Deletion of row $0$ and column $0$ in this matrix gives back the
matrix \eqref{bigmatptsinS}, whose determinant is $\ne0$. By
Proposition \ref{simplema}, therefore, the unique element (up to
scaling) in $V_R\oplus Rq=\spn(\tilde\ulp)_R$ with at least the given
zeros is this determinant, and it is not contained in $V_R$. But by
assumption, $f\in V_R$ is another nonzero element with at least
these zeros, contradicting uniqueness and thereby proving Proposition
\ref{genaunzeros}.
\end{proof}

\begin{lab}
Let $f$ span an extreme ray in $(P_{V,S})_R$, with $f$ as in
\ref{fwithnzeros}. As was just proved, $n$ is even and $f$ has
precisely $n$ zeros in $S_R$. Let these be $\xi_1,\dots,\xi_r$,
pairwise different, with respective (even) multiplicities
$b_1,\dots,b_r\ge2$, so $\sum_{i=1}^rb_i=n$. Write
$\ulb=(b_1,\dots,b_r)$ and $\ulxi=(\xi_1,\dots,\xi_r)$ and form the
matrix $Z=Z_\ulb(\ulp,\ulxi)$ over $A_R=A\otimes R$ as in
\eqref{bigmatbieq2}, with rows
$$\ulp,\ \ulp(\xi_1),\,\dots,\,\ulp^{(b_1-1)}(\xi_1),\,\dots,\,
\ulp(\xi_r),\,\dots,\,\ulp^{(b_r-1)}(\xi_r).$$
The determinant $\det(Z)$ is an element of $V_R$ with at least the
same zeros in $S_R$ as $f$. Choose any $\xi_0\in S_R$ that is not in
$\{\xi,\xi_1,\dots,\xi_r\}$. Then $\det(Z)$, evaluated at $\xi_0$, is
nonzero by Lemma \ref{n+1eval},
showing in particular that $\det(Z)$ is not identically zero. By
Proposition \ref{simplema}, therefore, $f=c\cdot\det(Z)$ with some
nonzero scalar $c\in R$. For calculating $\sosx f^\otimes(\xi_0)$
with $\xi_0\in S_R$,
we may assume $f=\det(Z)$.

Consider the matrix $T=T_{(1,\ulb)}(\ulp)$ with
$(1,\ulb):=(1,b_1,\dots,b_r)$, see \eqref{rowsofn}. So the blocks for
the Taylor process are
\begin{equation}\label{blox}%
B_0=\{0\},\ B_1=\{1,\dots,b_1\},\ B_2=\{b_1+1,\dots,b_1+b_2\}
\text{ etc},
\end{equation}
and $T$ is the $(n+1)\times(n+1)$ matrix over $A_r$ with rows
$$\varphi_0(\ulp),\ \varphi_1(\ulp),\,\dots,\,
\varphi_1(\ulp^{(b_1-1)}),\,\dots,\,\varphi_r(\ulp),\,\dots,\,
\varphi_r(\ulp^{(b_r-1)}),$$
up to nonzero scalar factors in~$\R$.
The matrix $Z$ arises from $T$ by applying the homomorphism
\begin{equation}\label{Ar2AotR}%
A_r\mapsto A\otimes R,\quad a_0\otimes\cdots\otimes a_r\>\mapsto\>
a_0\otimes\bigl(a_1(\xi_1)\cdots a_r(\xi_r)\bigr).
\end{equation}
By definition, the tensor evaluation of $f$ in $\xi_0\in S_R$ is the
image of $f=\det(Z)\in A\otimes R$ under
\begin{equation}\label{AotR2RotR}%
A\otimes R\to R\otimes R,\quad q\otimes a\mapsto q(\xi_0)\otimes a.
\end{equation}
So altogether, $f^\otimes(\xi_0)$ is the image of $\det(T)\in A_r$
under the ring homomorphism $A_r\to R\otimes R$,
\begin{equation}\label{Ar2RR}%
a_0\otimes\cdots\otimes a_n\>\mapsto a_0(\xi_0)\otimes
\bigl(a_1(\xi_1)\cdots a_r(\xi_r)\bigr).
\end{equation}
\end{lab}

\begin{lab}\label{endofpf}%
Theorem \ref{summazing} gives an expression \eqref{identsummazing}
for $\det(T)$, from which we get $f^\otimes(\xi_0)$ by applying the
homomorphism \eqref{Ar2RR} to the summands. Taking a typical summand
\begin{equation}\label{typsummand}%
\bfmu(t^\alpha)\cdot(1+h_{w,\alpha})\cdot\prod_{0\le i\le r}
\varphi_i(w')^{b_i(b_i-1)/2}\cdot\prod_{0\le i<j\le r}
\delta_{ij}(w)^{b_ib_j}
\end{equation}
in \eqref{identsummazing}, let us discuss its image under
\eqref{Ar2RR}, factor by factor:
\begin{itemize}
\item
$\bfmu(t^\alpha)=\bfmu(t_0^{\alpha_0}\cdots t_n^{\alpha_n})$ is
mapped to
$$t(\xi_0)^{\alpha_0}\otimes\Bigl(t(\xi_1)^{\alpha(B_1)}\cdots
t(\xi_r)^{\alpha(B_r)}\Bigr)$$
(with $B_1,\dots,B_r$ as in \eqref{blox} and
$\alpha(B_\nu)=\sum_{i\in B_\nu}\alpha_i$). By (S2) in
\ref{findes}, this is a tensor in $R\otimes R$ of the form
$a_1\otimes a_2$ with $a_1,\,a_2\ge0$.
\item
Let $\calO$ be the convex hull of $\R$ in $R$, a valuation ring of
$R$ with residue field $\R$. Let moreover $\omega\mapsto\ol\omega$
denote the natural residue map $\calO\otimes\calO\to\R$. The factor
$1+h_{w,\alpha}$ in \eqref{identsummazing} is mapped to an element
$\omega\in\calO\otimes\calO$ for which $\ol\omega>0$ in $\R$, by
(S1) in \ref{findes}. Such $\omega$ satisfies $\sosx(\omega)=1$,
according to \cite[Prop.~3.5]{sch:sxdeg}.
\item
$\varphi_i(w')$ is mapped to $w'(\xi_0)\otimes1$ (for $i=0$) or to
$1\otimes w'(\xi_i)$ (for $1\le i\le r$), respectively. In either
case, $w'(\xi_i)\ge0$ by (S3) in \ref{findes}.
\item
For $1\le i<j\le r$, $\delta_{ij}(w)^{b_ib_j}$ is mapped to an
element $1\otimes a$ with $a>0$, since $b_ib_j$ is even.
\item
\emph{The critical factor is} $\prod_{j=1}^r\delta_{0j}(w)^{b_j}$.
Under \eqref{Ar2RR} it is mapped to
\begin{equation}\label{mappd2}%
\prod_{j=1}^r\Bigl(w(\xi_0)\otimes1-1\otimes w(\xi_j)
\Bigr)^{b_j}.
\end{equation}
Since $b_1,\dots,b_r$ are even numbers with $\sum_{j=1}^rb_j=n$, this
is the square of a product of $\frac n2$ factors, each of the form
$w(\xi_0)\otimes1-1\otimes c_i$ with $c_i\in R$
($i=1,\dots,\frac n2$). Therefore, \eqref{mappd2} is equal to
$(a_0\otimes b_0+\cdots+a_{n/2}\otimes b_{n/2})^2$ with
$a_\nu,\,b_\nu\in R$.
In particular, the tensor \eqref{mappd2} has $\sosx\le1+\frac n2$.
\end{itemize}
All factors in the preceding list, except for the last, have
$\sosx\le1$. Their product therefore has $\sosx\le1+\frac n2$, see
\cite[Lemma~3.4(c)]{sch:sxdeg}. Since this holds for every summand
$w,\alpha$, it follows that
$$\sosx f^\otimes(\xi_0)\>\le\>1+\frac n2,$$
which finally completes the proof of the main theorem.
\qed
\end{lab}


\end{document}